\documentclass[12pt,a4paper,onecolumn,draft]{IEEEtran} 

\usepackage{amsmath}
\usepackage{amsthm,amsfonts,amssymb}
\usepackage{thmtools}
\usepackage[numbers]{natbib}
\usepackage{bbm}
\usepackage[fixamsmath]{mathtools}

\usepackage{graphicx}
\usepackage{caption}

\usepackage[colorlinks=true,breaklinks=true,bookmarks=true,urlcolor=blue,
     citecolor=blue,linkcolor=blue,bookmarksopen=false,draft=false]{hyperref}
\usepackage{cleveref}

\newcommand{\E}{\mathbb{E}}
\newcommand{\N}{\mathbb{N}}
\renewcommand{\P}{\mathbb{P}}
\newcommand{\R}{\mathbb{R}}

\newcommand{\cB}{\mathcal{B}}

\newcommand{\cF}{\mathcal{F}}

\newcommand{\cO}{\mathcal{O}}

\newcommand{\Ew}[1]{\mathbb{E}\left[#1\right]}
\renewcommand{\Pr}[1]{\mathbb P\left(#1\right)}
\newcommand{\norm}[1]{\lVert #1 \rVert}
\newcommand{\abs}[1]{\lvert #1 \rvert}

\newcommand{\as}{\text{ almost surely}}
\newcommand{\io}{\text{ i.o.}}
\newcommand{\lest}{\le_{\text{st}}}
\newcommand{\convh}{\overline{\text{co}}}

\newcommand{\interior}[1]{%
  {\kern0pt#1}^{\mathrm{o}}%
}

\theoremstyle{plain}
\newtheorem{theorem}{Theorem}
\newtheorem{lemma}{Lemma}
\newtheorem{corollary}{Corollary}

\newtheorem*{claim*}{Claim}

\theoremstyle{remark}
\newtheorem{assumption}{Assumption}[section]
\newtheorem{definition}{Definition}
\newtheorem{remark}{Remark}

\declaretheorem[style=definition]{example}
\renewcommand\thmcontinues[1]{Continued}

\title{Stability and Convergence of Distributed Stochastic Approximations with large Unbounded Stochastic Information Delays}
\author{Adrian Redder, Arunselvan Ramaswamy, Holger Karl}

\begin{document}

\maketitle

\begin{abstract}
    We generalize the Borkar-Meyn stability Theorem (BMT) to distributed stochastic approximations (SAs) with information delays that possess an arbitrary moment bound. To model the delays, we introduce Age of Information Processes (AoIPs): stochastic processes on the non-negative integers with a unit growth property. We show that AoIPs with an arbitrary moment bound cannot exceed any fraction of time infinitely often. 
In combination with a suitably chosen stepsize, this property turns out to be sufficient for the stability of distributed SAs. 
Compared to the BMT, our analysis requires crucial modifications and a new line of argument to handle the SA errors caused by AoI. In our analysis, we show that these SA errors satisfy a recursive inequality. To evaluate this recursion, we propose a new Gronwall-type inequality for time-varying lower limits of summations. As applications to our distributed BMT, we discuss distributed gradient-based optimization and a new approach to analyzing SAs with momentum.
\end{abstract}



\section{Introduction}

Stochastic approximations (SAs) have been of renewed interest 
since the seminal contributions by Robbins \& Monro \cite{robbins1951stochastic} and Kiefer \& Wolfowitz \cite{kiefer1952stochastic}  due to various applications in signal processing, economics, game theory, machine learning, and optimization \cite{benveniste2012adaptive, bravo2016adjusted,lei2020synchronous,s2013stochastic, uryasev2013stochastic,ghadimi2012optimal}. Traditionally, SA is a centralized paradigm implemented as a single iteration. Centralized paradigms, however, typically suffer from computational bottlenecks or are infeasible due to the decentralized nature of a problem. Therefore, distributed asynchronous SA algorithms have been developed where multiple systems/nodes/agents interact with each other to solve a global SA problem. Such distributed asynchronous parallel implementations of SA algorithms were first considered for stochastic gradient-based methods \cite{tsitsiklis1986distributed}. An extensive collection of distributed parallel algorithms can be found in \cite{bertsekas2015parallel}. 

Distributed SA refers to algorithms that are executed via computer networks and are thus affected by communication delays. Asynchronous SA traditionally refers to iterations that run with different clocks, such that at every time step only some iterations are updated \cite{borkar1998asynchronous}. Recently, asynchronous SA has also been used to refer to algorithms where a global variable is updated by a set of workers (computing nodes). In these asynchronous computing scenarios, asynchronous updates lead to information delays from the perspective of the global variable iteration \cite{lian2015asynchronous}. 
The advantage of this approach is the potentially significantly enhanced convergence speed \cite{zhang2013asynchronous}. The problem is that the resulting iteration is affected by errors due to the information delays, which can lead to loss of convergence guarantees or even instability.

Beyond computing scenarios, information delays arise in distributed SA scenarios due to the decentralized nature of physical multi-agent systems. Such systems usually have to run distributed algorithms under limited communication resources, for instance, remote battery-powered wireless sensor networks \cite{he2020distributed}. In the most extreme scenario, such systems must withhold the exchange of information for as long as possible  to minimize their power consumption. The natural question here is how little information sharing can be allowed so that distributed systems can solve SA problems in a decentralized manner.

The dominant feature in the aforementioned computing and optimization scenarios is that a set of variables is updated as a function of old values of itself. It therefore follows that the resulting iterations are affected by \emph{Age-of-Information (AoI)}. \emph{With this work, we intend to close a chapter on SA algorithms, namely, distributed SAs with large unbounded stochastic information delays.} Specifically, we give sufficient conditions for the stability and convergence of distributed SAs in the presence of AoI.
As an illustration, consider the following distributed iteration $x_n \coloneqq (x^1_n, \ldots x^D_n)$ in 
$\R^d\coloneqq\R^{d_1} \times \ldots \times \R^{d_D}$:
\begin{equation}
\label{eq: main_iteration_example}
   x^i_{n+1} = x^i_n + a(n) f^i(x^1_{n-\tau_{i1}(n)}, \ldots, x^D_{n-\tau_{iD}(n)}, \xi^i_n), \quad n\ge1, \quad (1\le i \le D),
\end{equation}
where  
\begin{enumerate}
    \item $\tau_{ij}(n)$ are delay/AoI random variables that incur since iteration $i$ uses the value of iteration $j$ from the old time step $n-\tau_{ij}(n)$ to evaluate its local drift/dynamics $f^i$ at time $n$.  
    \item $\xi^i_n$ are random samples from a sample space $\Xi$.
    \item $\{a(n)\}$ is a sequence of positive numbers referred to as the algorithm stepsize.
\end{enumerate}
Under very mild conditions on the drift functions $f^i$ 
we will show that \eqref{eq: main_iteration_example} is stable and converges almost surely to compact connected invariant set of the ODE $ \dot{x}(t) = \E_{\xi}[f(x(t),\xi)]$ provided there exists \emph{an arbitrary} $p>0$, such that $\sup_{n\ge 0} \Ew{\tau^p_{ij}(n)} < \infty$.
As of now, this was only known if either (a) $x_n$ is prematurely assumed to be stable almost surely and the above condition holds for at least $p>1$ \cite[Section 6]{borkar2009stochastic}, or (b) the AoI variables $\tau_{ij}(n)$ are almost surely bounded \cite{bhatnagar2011borkar}. 

\subsection{Main Contribution}

The main contribution of this paper is the stability and convergence of distributed SA algorithms in the presence of AoI. The iterations that we consider are of the following form:
\begin{equation}
    \label{eq: main_rewritten_intro}
    x_{n+1} = x_n + a(n)\left[h(x_n) + e_n + M_{n+1} \right], \quad n\ge1,
\end{equation}
where $h$ is a function that characterizes the mean algorithm dynamics/drift (e.g. $h(x) = \E_{\xi}[f(x,\xi)]$ for iteration \eqref{eq: main_iteration_example}), $e_n$ is a sequence of drift errors due to the AoI $\tau_{ij}(n)$ and $M_{n+1}$ is a martingale difference noise sequence due to the use of samples. The standard procedure to analyze such iterations is to first establish (or assume) its stability, i.e. that $\sup_{n\ge0} \norm{x_n} < \infty$. Then one verifies that the iteration converges assuming stability.
There are various schemes to establish the stability of SA algorithms, see \cite{benaim1996dynamical}, \cite[Chapter 4]{borkar2009stochastic} and the reference therein. One of the most remarkable schemes is the stability through scaling approach proposed in \cite{borkarmeyn2000ode}, which is now known as the Borkar-Meyn Theorem (BMT). The stability is shown by studying a family of ordinary differential equations (ODEs) with scaled dynamics $\frac{h(c x)}{c}$ for $c\in [1,\infty]$. This scheme is attractive as its assumptions can be verified solely using the algorithm drift $h$. Other schemes in the literature are often problem-specific and are, for instance, based on an available Lyapunov function \cite{harold1997stochastic}. 

The BMT considers iteration \eqref{eq: main_rewritten_intro} with zero drift errors, i.e. with  $e_n = 0$. Further, it is  straightforward to show that the BMT also holds when $e_n$ is a deterministic or random bounded sequence in $o(1)$. \emph{On the other hand, it is very far from obvious that the drift errors $e_n$ caused by AoI converge to zero almost surely when $x_n$ is not assumed stable.} Our main contribution is a distributed BMT, \Cref{thm: main_distributed}, for exactly such scenarios. Our derivation carefully accounts for the errors due to AoI. Two key contributions, also of independent interest, are required to show the distributed BMT:
\begin{enumerate} 
    \item We introduce Age-of-Information Processes (AoIPs) to model the information delay processes $\tau_{ij}(n)$. An AoIP is a stochastic process on the non-negative integers with a unit growth property. Delay processes that capture the newest available information from a source at a monitor are naturally AoIPs. We show asymptotic growth properties for AoIPs as a function of suitable moment bounds in \Cref{lem: aoi_bound}. 
    With this, we show in \Cref{lem: step_size_AoI_sum} that the SA stepsize accumulated over intervals with AoI length $\tau_{ij}(n)$ converges to zero almost surely, i.e. that $\sum_{k=n-\tau_{ij}(n)}^{n-1} a(k)  \to 0 \as$. This convergence property, which relates the AoIPs and the SA stepsize decay, is identified as the key sufficient condition to establish the stability and almost sure convergence of distributed SAs.

    \item We propose a new Gronwall-type inequality, \Cref{lem: new_gronwall}, to bound iterations that satisfy a linear recurrence inequality with varying lower time horizons. We identified that the SA errors due to AoI satisfy such recursive inequalities both in norm as well as in $L_2$, these inequalities are then evaluated using the new Gronwall-type inequality.
\end{enumerate}
In addition to these components, our distributed BMT requires crucial modifications compared to the traditional BMT to handle the algorithm drift errors caused by AoI. With the distributed BMT, we provide for the first time a set of sufficient conditions for the stability and convergence of distributed SAs with information delays that merely possess an arbitrary uniform moment bound.


\subsection{Further Results}
Beyond our main stability theorem, our analysis has also led to new insights into the traditional BMT. 
Our new line of argument for the distributed BMT applies to the original BMT and has revealed how to weaken a key assumption in the original BMT: the original BMT requires that $\frac{h(c x)}{c}$ converges pointwise to some limit $h_\infty(x)$ as $c \to \infty$, where $h_\infty$ is globally asymptotically stable to the origin. We show that it is merely required that a scaling sequence $c_n \nearrow \infty$ exists such that $\lim\limits_{n \to \infty} \frac{h(c_{n}x)}{c_{n}}$ is globally asymptotically stable to the origin (\Cref{thm: genBMT}). Previously, the limit needed to exist for any scaling sequence drifting to infinity.

Finally, beyond natural applications to distributed optimization (\Cref{sec: app1}), our tools also enable the study of stochastic approximations with momentum (\Cref{sec: app2}). Consider the following SA iteration with Polyak's heavy ball momentum:
\begin{equation}
\begin{split}
     \label{eq: momentum_intro}
    x_{n+1} &= x_n + a(n) m_k \\
    m_k &= \beta m_{k-1} + (1-\beta)g(x_k)
\end{split} 
\end{equation}
with $m_0 = 0$, $\beta \in [0,1)$ and $g(x_k) \coloneqq h(x_k) + M_{k+1} $. This iteration has been extensively for stochastic gradient descent with momentum, but in general only for specific SA iterations or for linear SA iterations, whereby always a momentum parameter $\beta_n \nearrow \infty$ has been chosen. See \Cref{sec: app2} for a discussion of the relevant related work on SAs with momentum.

We observed that \eqref{eq: momentum_intro} can be studied by splitting the moving average of the past drift terms into two contributions ``new'' and ``old'' drift terms. Specifically,   \eqref{eq: momentum_intro} can be written in moving average form as 
\begin{equation}
    x_{n+1} = x_n + a(n) (1-\beta)\left[\sum_{i=1}^n \beta^{n-i} g(x_i) \right]
\end{equation}
Now define a deterministic AoI sequence $\tau(n) \coloneqq \lceil \frac{n}{\log(n+1)}\rceil$ and split the above summation into two components:
\begin{equation}
\label{eq: momentum_intro_rewritten}
    x_{n+1} = x_n + a(n) (1-\beta)\left[\sum_{i=n-\tau(n)}^n \beta^{n-i} g(x_i) \right] +  a(n) (1-\beta)\left[\sum_{i=1}^{n-\tau(n)-1}\beta^{n-i} g(x_i) \right].
\end{equation}
Under standard assumptions for the drift $h$ and the martingale difference noise $M_{n+1}$, we will show that the second summation, which averages ``old'' drifts, is in $o(1)$. The resulting iteration can then be studied as an iteration affected by AoI along the lines of our distributed BMT. Specifically, we will conclude that the BMT also holds for heavy-ball stochastic approximations \eqref{eq: momentum_intro} and we can therefore provide sufficient conditions for stability and convergence of SAs with heavy-ball momentum (\Cref{thm: momentum}). This is a new result for general SA iteration. See \Cref{sec: app2} for details.

\subsection{Related Work}

If not otherwise stated, all of the following works consider decaying stepsizes that satisfy the Robbins and Monro conditions, i.e. the stepsizes are not summable but square summable.

One of the earliest work on asynchronous distributed SA date back to \cite{bertsekas1982distributed}, where an abstract dynamic programming approach was proposed in the presence of bounded communication and computing delays. In \cite{tsitsiklis1994asynchronous} the first asynchronous distributed SA algorithm for settings with potentially unbounded delay was proposed. The assumptions are tailored towards Q-learning and it was required that all delays satisfy that $n-\tau_{ij}(n) \to \infty$ almost surely. The first asynchronous distributed SA of the form \eqref{eq: main_iteration_example}, which from today's point of view is in the standard form of an SA iteration, was considered in \cite{borkar1998asynchronous}. Here, almost sure convergence is shown assuming stability and that the delays satisfy, for some $p>0$,
\begin{equation}
\label{eq: conditional_moment_bound}
    \sup_{n\ge 0} \Ew{\tau^p_{ij}(n) \mid \tau_{ij}(k),  k \le n-1} < \infty.
\end{equation}
This assumption is quite restrictive compared to the unconditional version assumed in this paper. For example, even a simple random walk with resets on the positive integers, \cite{redder22allerton}, will not satisfy \eqref{eq: conditional_moment_bound}. This is because \eqref{eq: conditional_moment_bound} basically requires a delay moment bound independent of the delay at the previous time step. In \cite[Chapter 6]{borkar2009stochastic} it was then shown that $\sup_{n\ge 0} \Ew{\tau^p_{ij}(n)} < \infty$ for some $p>1$ is sufficient for almost sure convergence to an equilibrium, still assuming stability. 
In \cite{bhatnagar2011borkar} the stability of general asynchronous SAs was then considered for bounded delays. Specifically, a version of the BMT was shown to hold under bounded delays and slightly stronger martingale noise assumptions than typical. We will not require any of those restrictions. 
The key insights that allow us to make this progress are the disclosure of 1. the crucial interplay between the algorithm AoI and algorithm stepsize, and 2. the recursive structure of the SA drift errors caused by AoI. Finally, it should be noted that the analysis in \cite{bhatnagar2011borkar} also considers the aforementioned traditional asynchronous updates, which can be included in our framework under the assumptions presented therein.

Beyond general stochastic approximation iterations, there have been several works on distributed gradient-based methods with delays. Distributed SGD with bounded information delays were first considered in the seminal work of \cite{tsitsiklis1986distributed}. Here it was sketched for the first time that delays may be allowed to grow sublinearly relative to a global clock when a sufficiently rapid decaying stepsize is chosen. 
Finite time error bounds for asynchronous SGD algorithms under convex stochastic objectives, constant stepsizes and bounded delays were proposed in \cite{agarwal2011distributed} and \cite{feyzmahdavian2016asynchronous}. Finite time bounds for the mean square variation of the mean gradient of SGD under a time-varying stepsize were proposed in
\cite{lian2015asynchronous} for general non-convex objectives and bounded delays.
Almost sure convergence of SGD to stationary points under merely locally Lipschitz continuous gradients with noise-dependent Lipschitz constants was proven in \cite{ramaswamy2021distributed}. However, stability was assumed as well as the delay conditions proposed in \cite[Chapter 6]{borkar2009stochastic}.

As far as delays are concerned the closest to our work is \cite{zhou2022distributed}. The delays considered therein are potentially large and unbounded but are assumed as deterministic. We claim that it is more representative to work with stochastic delays. However, stochastic delays lead to SA error bounds that are summations over random time horizons, see \Cref{sec: main_results}.

The algorithm considered in \cite{zhou2022distributed} is asynchronous SGD, i.e. multiple workers are computing updates for a single global variable. The authors focus on two scenarios, general non-convex objectives and variational coherent objectives. 
The second scenario is not the scope of our paper, and we instead compare it to the first scenario, which considers asynchronous SGD without projections for an unconstrained non-convex optimization problem with objective $h(x) \coloneqq \E_\xi[f(x;\xi)]$. The paper shows that $\lim\limits_{n\to \infty} \Ew{\norm{\nabla_x h(x_n)}^2} = 0$, i.e. that the gradient converges in mean square, where here the expectation is with respect to $x_n$. For the same scenario, we provide conditions that show that $ \norm{\nabla_xh(x_n)} \to 0 \as$. This provides a stronger characterization for every individual trajectory of the stochastic iteration and shows that practically every instantiation of the algorithm converges to a critical point. Furthermore, our analysis holds under different assumptions. Both analyses require that $\E_\xi[\nabla_xf(x;\xi)]$ is Lipschitz continuous. The analysis in \cite{zhou2022distributed} requires that $\sup_{x \in \R^d} \E_\xi[\norm{\nabla_x f(x; \xi)}^2] < \infty$, which implies that  $\sup_{x\in\R^d} \norm{\nabla_x h(x)} < \infty$. We only require that $\E_\xi[\norm{\nabla_x f(x; \xi)}^2] \le K(1+\norm{x}^2$)  for all $x\in \R^d$ for some $K>0$ and we thus even allow the objective gradient $\nabla_x h(x)$ to be unbounded. However, to apply our distributed BMT we require in addition that $h$ basically acts like a convex function when scaled to infinity. This growth condition is naturally satisfied for quadratic objectives and can be easily guaranteed by $L_2$ regularization. Our analysis, therefore, covers for the first time the stability and convergence of SGD for a class of quadratic stochastic objectives in the presence of large unbounded stochastic delays. We will discuss the details to this application in \Cref{sec: app1}.

\subsection{Methods}


Our stability analysis is inspired by the traditional BMT, i.e. the stability is obtained through a scaling limit of the SA dynamics $h$. 
We highlight all key differences in \Cref{sec: outline}, where we outline our stability analysis.
Both stability and convergence are obtained through ODE techniques \cite{benaim1996dynamical}, i.e. discrete iterations are shown to closely track solutions to certain ODEs.  We would like to point out that it was mentioned in \cite{zhou2022distributed} that the almost sure convergence of stochastic approximations with unbounded delays might be difficult to achieve via an ODE approach (as to be presented in this work). The reason for this is that it was hypothesized that without the assumption of stability, the large delays would separate the SA iteration from its continuous-time counterpart. As we will see, the convergence property $\sum_{k=n-\tau_{ij}(n)}^{n-1} a(k)  \to 0 \as$, relating delays and algorithm stepsize, is sufficient to connect these two worlds. We can therefore provide for the first time sufficient conditions for the convergence of distributed stochastic approximations with unbounded delays assuming neither that the iterations are stable nor that the iterations have bounded dynamics.


\section{Problem Setup and Main Results}
\label{sec: setup}

Throughout this paper, we focus on the following generalization of iteration \eqref{eq: main_iteration_example}:
\begin{equation}
\label{eq: main_iteration}
   x^i_{n+1} = x^i_n + a(n) \left[ h^i(x^1_{n-\tau_{i1}(n)}, \ldots, x^D_{n-\tau_{iD}(n)}) + M^i_{n+1}\right], \quad n\ge1, \quad (1\le i \le D),
\end{equation}
where $h^i: \R^d \to \R^{d_i}$ and $M^i_{n+1}$ are the local drift functions and local additive noise terms of iteration $i$, respectively. Further, $\tau_{ij}(n)$ are delay random variables that incur since iteration $i$ uses the iteration value of iteration $j$ from time $n-\tau_{ij}(n)$ to evaluate its drift term at time $n$.
In standard terminology, \eqref{eq: main_iteration} is called a distributed stochastic approximation scheme. If the errors 
\begin{equation}
    \label{eq: local_drift_errors}
    e^i_n \coloneqq h^i(x^1_{n-\tau_{i1}(n)}, \ldots, x^D_{n-\tau_{iD}(n)}) -h^i(x^1_n,\ldots ,x^D_n)
\end{equation}
satisfy that $e^i_n \in o(1)$ almost surely, then we expect that \eqref{eq: main_iteration} tracks solutions to the ODE
\begin{equation}
    \label{eq: ODE_method}
    \dot{x}(t) = h(x(t))
\end{equation}
under suitable assumptions on $a(n)$, $h$ and $M^i_{n+1}$. This is the dynamical systems perspective (also known as the ODE method) of SA. The standard regularity assumption to ensure that the ODE \eqref{eq: ODE_method} is well-posed is that $h$ satisfies a Lipschitz condition, which we will make for the rest of the paper:
\begin{assumption}
\label{asm: lipschitz}
    $h^i : \R^d \to \R^{d^i}$ is the $i$-th component of a Lipschitz-continuous map $h: \R^d \to  \R^d$  with Lipschitz constant $L>0$.
\end{assumption}
The Lipschitz condition will play a crucial role to establish the stability and convergence of \eqref{eq: main_iteration} in the presence of the errors $e^i_n$. In addition, we require that rescaled versions of the ODE \eqref{eq: ODE_method} converge to an ODE with a globally asymptotically stable equilibrium.
\begin{assumption}
\label{asm: BMT}
    The functions $h_c(x) \coloneqq \frac{h(cx)}{c}, c\ge 1$, $x\in \R^d$, satisfy $h_c(x) \to h_\infty(x)$ pointwise as $c\to \infty$ for some $h_\infty \in C(\R^d)$. Furthermore, the ODE
    \begin{equation}
    \label{eq: c_infty_ODE}
        \dot{x}(t) = h_\infty(x(t)) 
    \end{equation}
    has the origin as its unique globally asymptotically stable equilibrium.
\end{assumption}
\Cref{asm: BMT} originates from the traditional Borkar-Meyn Theorem (BMT), see \cite{borkarmeyn2000ode} or \cite[Chapter 3]{borkar2009stochastic}. As mentioned before, it is one of our main contributions to extend the BMT to a distributed setting with unbounded stochastic delays. Moreover, we will present a weaker version of \Cref{asm: BMT} in \Cref{sec: genBMT} inspired by the analysis of our distributed BMT.

For the additive noise terms $M^i_{n+1}$, we make another standard assumption from the SA literature rephrased to suit the distributed SA setting:
\begin{assumption}
    \label{asm: noise}
    $M^i_{n+1} \in \R^{d^i}$ is the $i$-th component of a martingale difference noise process $\{M_n\}_{n\ge1}$ in $\R^d$ with respect to the filtration 
    \begin{equation}
        \cF_1 = \sigma(x_1),\ \cF_n \coloneqq \sigma\left(x_1, M_2, \ldots M_n, \tau_{ij}(1), \ldots \tau_{ij}(n), 1\le i,j \le D\right),\quad n\ge2:
    \end{equation}
    \begin{enumerate}
        \item $\Ew{M^i_{n+1} \mid \cF_n} = 0$.
        \item $\Ew{\norm{M^i_{n+1}}^2 \mid \cF_n} \le K^2\left(1 + \norm{(x^1_{n- \tau_{i1}(n)}, \ldots, x^D_{n-\tau_{iD}(n)})}^2 \right)$ for some $K>0$.
    \end{enumerate}
\end{assumption}
\begin{remark}
    \Cref{asm: noise} bounds the conditional second moment of the martingale difference noise components at time $n\ge 1$ based on the associated iteration values in \eqref{eq: main_iteration} that are used to evaluate the drift components at time $n\ge1$. This appears to be the most useful version since the noise is often a function of the iteration value used at the corresponding time step. A more general version that might also be of use is \Cref{sec: app2}:
    \begin{equation}
        \label{eq: weaker_martingale}
        \Ew{\norm{M^i_{n+1}}^2 \mid \cF_n} \le K^2\left(1 + \sup_{0 \le k_{ij} \le \tau_{ij}(n)}\norm{(x^1_{n- k_{i1} }, \ldots, x^D_{n-k_{iD}})}^2 \right) \text{ for some } K>0.
    \end{equation}
    Our analysis, notably \Cref{lem: error_L2_bound}, \Cref{lem: main_l2_bound} and \Cref{lem: martingale_convergence} hold with minor modifications with \eqref{eq: weaker_martingale} instead of \Cref{asm: noise}.2. We use \Cref{asm: noise}.2 to simplify the presentation. 
\end{remark}
Finally, we will now state our assumptions for the stepsize sequence $a(n)$ and the age processes $\tau_{ij}(n)$. These processes can be the consequence of various transport phenomena that result in the use of aged information $x^j_{n-\tau_{ij}(n)}$. We call these random variables aged since they are old versions of the current variables $x^j_n$. To analyze the $\tau_{ij}(n)$'s, we define what we call an age of information process. 

\subsection{Age of Information Processes}

 \begin{definition}
\label{def:AoI_process}
An \textbf{Age of Information Process (AoIP)} is a discrete-time stochastic process $\tau(n)$ on the non-negative integers that is associated with a system containing two components: an information stream, often modeled as a stochastic process $X=\{X_n\}_{n\ge1}$, and a monitor that processes this information stream. The AoIP $\tau(n)$ captures the age of the 
information from $X$ that is available to the monitor for further processing at time $n$. Specifically, when samples $X_{m(1)}, \ldots, X_{m(n)}$ from time steps $0 \le m(1) < \ldots < m(n) \le n$ are available to the monitor at time $n$, then the AoI at time $n$ is defined as
\begin{equation}
    \tau(n) \coloneqq n - m(n).
\end{equation}
By construction an AoIP has two natural properties: 1. $\tau(n+1) \le \tau(n) + 1$ (unit growth property) and
2. $n - \tau(n)$ is monotonically increasing.
\end{definition}

In distributed multi-agent systems each agent is usually in correspondence with a local process $X^i$ for every agent~$i$. One can now naturally associate two AoIPs $\tau_{ij}(n)$ and $\tau_{ji}(n)$ with every pair of agents $(i,j)$.
Note that $\tau_{ii}(n) = 0$, $1\le i \le D$, is typical for communication scenarios, but it is not necessary. 

We now state our last two assumptions, for which we use the notion of stochastic dominance.
\begin{definition}
\label{def:stochastic_dominance}
A non-negative random variable $\tau$ is said to be \emph{stochastically dominated} by a random variable $\overline{\tau}$, denoted by $\tau \lest \overline{\tau}$, if  $\Pr{\overline{\tau} > m} \ge \Pr{\tau > m } $ for all $m\ge0$.
\end{definition}
\begin{assumption}
\label{asm: AoI}
    $\tau_{ij}(n)$ are AoIPs, such that there exist some $p>0$ and a random variable $\overline{\tau}$ with $\tau_{ij}(n) \lest \overline{\tau}$, for all $n\ge1$ and $1\le i,j \le  D$,  such that
    $\Ew{\overline{\tau}^p} < \infty.$
\end{assumption}
\begin{assumption}
\label{asm: stepsize}
    $a(n) > 0$ for all $n\ge1 $, is a stepsize sequence, such that
    \begin{enumerate}
        \item $a(n) = \frac{a}{n}$, $n\ge1$, if $p \in (0,1]$
        \item $a(n) = an^{-\frac{1}{q}}$, $n\ge1$, with $1 < q < \min\{2,p\}$, if $p> 1$.  
    \end{enumerate}
    for some $a>0$ and with $p >0$ from \Cref{asm: AoI}.
\end{assumption}
\Cref{asm: AoI} and \Cref{asm: stepsize} formulate a trade-off between the AoIP moments and the decay of the stepsize sequence $a(n)$; a faster uniform tail decay of the AoI distributions allows slower decaying stepsize, thence potential faster convergence. Specifically, if \Cref{asm: AoI} holds with $p>1$ then stepsizes that decay slower than $\frac{1}{n}$ become available. \emph{Notably, the standard stepsize $a(n) = \frac{1}{n}$, $n \ge 1$, turns out to be sufficient for stability and convergence of \eqref{eq: main_iteration} under an arbitrary AoIP moment bound and most importantly for $p\in (0,1]$.} This was not known before this work. 

\begin{remark}
    \Cref{asm: stepsize} guarantees that the stepsize $a(n)$ is not summable but square summable, i.e.: $\sum_{n\ge1} a(n) = \infty$ and $\sum_{n \ge1} a(n)^2 < \infty.$ 
We only state the stepsizes in  \Cref{asm: stepsize} to simplify the presentation. All proofs hold provided that the stepsizes are not asymptotically larger than the stepsizes presented in \Cref{asm: stepsize}. E.g., for $p\in (0,1]$ we only require $a(n) \in \cO\left(\frac{1}{n}\right)$ in \Cref{asm: stepsize}.1. 
\end{remark}
\begin{remark}
It is always possible to construct a random variable that stochastically dominates a finite set of random variables provided that stochastically dominating random variables are given for each random variable of the finite set. Hence, assuming a single dominating random variable for all AoIPs as formulated in \Cref{asm: AoI} is without loss of generality.
\end{remark}
\begin{remark}
    AoIPs can be combined to create more complex AoI structures where multiple sources give rise to AoI at a monitor. For example, a more general class of age/delay processes can be defined as stochastic processes that take at every time step the value of one AoIP from a finite family of AoIPs. Our results for AoIPs to be presented in the next subsection extend to such combined AoIPs.
\end{remark}

\subsection{Main Results}

We will now state our main stability theorem and its convergence corollary. The detailed proofs are presented in \Cref{sec: main_results}. Below we will present self-contained results for AoIPs that will be important to control the effect of the local drift errors $e^i_n$ in \Cref{sec: main_results}. After that, we present an outline of our stability analysis.

\begin{theorem}[Distributed Borkar-Meyn Theorem]
    \label{thm: main_distributed}
    Under \Cref{asm: lipschitz}-\ref{asm: stepsize}, $\sup_n \norm{x_n} < \infty \text{ a.s.}$
\end{theorem}
\begin{corollary}
    \label{cor: convergence}
    Iteration \eqref{eq: main_iteration} converges almost surely to a potential sample path-dependent compact connected internally chain transitive invariant set of the ODE \eqref{eq: ODE_method}.
\end{corollary}

The present paper was prompted in large part by the following lemma, \Cref{lem: aoi_bound}, and its successor, \Cref{lem: step_size_AoI_sum}, both of which are of independent interest. 
\begin{lemma}
\label{lem: aoi_bound}
Let $\tau(n)$ be an AoIP. Suppose there exists a random variable $\overline{\tau}$ with $\tau(n) \lest \overline{\tau}$ for all $n\ge1$ with $\Ew{\overline{\tau}^p}$ for some $p\in (0,\infty)$, then 
\begin{enumerate}
    \item $\Pr{\tau(n) > \varepsilon n  \io} = 0, \quad \text{ if } p \in (0,1]$.
    \item $\Pr{\tau(n) > \varepsilon n^{\frac{1}{p}} \io } = 0, \quad \text{ if } p > 1.$
\end{enumerate}
for all $\varepsilon \in (0,1)$.
\end{lemma}
\Cref{lem: aoi_bound} shows that if an AoIP has at least a dominating random variable with some moment bound, then the AoIP will not exceed any fraction of $n$ after some time. To show this, we use the first Borel-Cantelli Lemma and the stochastic dominance property to conclude that the AoIP has a subsequence that does not exceed any fraction of $n$ after some time. We will then use the unit growth property of AoIPs to show that the AoIP does not exceed any fraction of $n$ plus a term in the order of $n^{1-p}$ after some time. A simple argument based on the limit superior then completes the proof (see Appendix \ref{sec: appendix} for details.)

The key consequence of \Cref{lem: aoi_bound} is that under \Cref{asm: AoI} and \Cref{asm: stepsize}, the stepsize accumulated over intervals $[n-1, n-\tau_{ij}(n)]$ converges to zero almost surely. The lemma, therefore, connects the uniform tail decay of AoIPs with the algorithm stepsize decay.

\begin{lemma}
Under \Cref{asm: AoI} and \Cref{asm: stepsize} it follows that
\label{lem: step_size_AoI_sum}
$$\sum_{k=n-\tau_{ij}(n)}^{n-1} a(k) \to 0 \as.$$
\end{lemma}

\begin{proof}
We drop the indices $i$ and $j$ in $\tau_{ij}(n)$ for this proof.
Consider $p\in(0,1]$, i.e. we pick $a(n) = \frac{a}{n}$. Since $a(n)$ is monotonically decreasing it follows that
\begin{align}
    \sum_{k=n-\tau(n)}^{n} a(k) \le \sum_{k=n-\tau(n)}^{n} \frac{a}{k} &\le \frac{a}{n-\tau(n)}  + \int_{n-\tau(n)}^n \frac{a}{t} dt     \\ &= \frac{a}{n-\tau(n)} + a \log\left( \frac{n}{n-\tau(n)} \right).
\end{align}
Then the lemma follows by continuity of the logarithm provided that
\begin{equation}
    \frac{n-\tau(n)}{n} \to 1 \as \iff \frac{\tau(n)}{n} \to 0 \as
\end{equation}

To show this, we state a lemma that is often used to prove the strong law of large numbers. Here, we restate it as required for the present setting.
\begin{lemma}[{\cite[Lemma 5.2.2.]{rosenthal2006first}}]
\label{lem: as_conv_lemma}
Let $X_1,X_2, \ldots$ be real-valued random variables. Suppose for each $\varepsilon>0$, we have $\Pr{\abs{X_n} \ge \varepsilon \io } = 0.$ Then $X_n$ converges to 0 almost surely.
\end{lemma}

Fix $\varepsilon \in (0,1) $.
By (A1) there is a random variable $\overline{\tau}$ with $\tau(n) \lest \overline{\tau}$ and $\Ew{\overline{\tau}^p}< \infty$. \Cref{lem: aoi_bound} then shows that 
\begin{align}
    \Pr{\tau(n) > \varepsilon n \io }= 0 
    \iff \quad  \Pr{\frac{\tau(n)}{n} > \varepsilon \io } = 0
\end{align}
As this holds for every $\varepsilon \in (0,1)$, it follows from the above lemma that $\frac{\tau(n)}{n} \to 0 \as$, which thus implies the statement to be shown. The case $p>1$ follows a similar line of argument and is relegated to  Appendix \ref{sec: appendix}.
\end{proof}

To understand the importance of \Cref{lem: step_size_AoI_sum} for the stability of distributed SAs, we will now present an outline of our stability proof. Notably, it turns out that the almost sure convergence in \Cref{lem: step_size_AoI_sum} reaches deeper into our stability analysis than one might initially expect.

\subsection{Stability Analysis Outline}
\label{sec: outline}

Our stability analysis is inspired by the traditional BMT and can be seen as a generalization to a distributed BMT. 
To better understand the required changes, we first sketch the proof of the traditional BMT, i.e. we consider \eqref{eq: main_iteration} with $\tau_{ij}(n) = 0$. 

\emph{Traditional BMT outline}:
First, create a piecewise linear interpolated trajectory $\overline{x}(t)$ from $x_k$. Then, separate the time axis $[0, \infty)$ into concatenated time segments $[T_m, T_{m+1}]$ of length approximately $T > 0$.
Next, create a rescaled trajectory $\hat{x}(t)$ by dividing $x(t)$ over each
$[T_m, T_{m+1}]$ by $\norm{x(T_m)}$, i.e. each $x(T_m)$
is scaled/projected to the unit ball. It is then easy to show that $\hat{x}(t)$ is stable almost surely and one can conclude by a stochastic approximation argument that $\hat{x}(t)$ tracks solutions to scaled ODEs with drift $h_c(\cdot)$, $c\in [1, \infty]$, from \Cref{asm: BMT}. The stability proof then follows by contradiction. Assuming that $x_n$ is unstable, there will be a subsequence of scaling factors $\norm{x(T_m)}$ diverging to infinity. A stochastic approximation argument then leads to corresponding rescaled segments of $\hat{x}(t)$ that asymptotically track the limiting ODE \eqref{eq: c_infty_ODE} with drift $h_\infty(\cdot)$ from \Cref{asm: BMT}. Since this limiting ODE is globally asymptotically stable to the origin, it follows that the aforementioned rescaled segments eventually drift towards the origin, which leads to a contradiction. We will now outline the proof of our distributed BMT.

\emph{Distributed BMT outline}: Rewrite the main iteration \eqref{eq: main_iteration} as
\begin{equation}
\label{eq: main_iteration_with_error}
   x_{n+1} = x_n + a(n) \left[ h(x_n) + e_n + M_{n+1}\right],
\end{equation}
for $n\ge1$, with additive drift error $e_n = (e^1_n, \ldots, e^D_n)$ and local drift errors $e^i_n$ as defined in \eqref{eq: local_drift_errors}. Equation \eqref{eq: main_iteration_with_error} has the form of a standard stochastic approximation iteration with additive drift error $e_n$ and martingale difference noise $M_{n+1}$. Our main task is to verify that $x_n$ is stable almost surely. Ones this is shown, in view of \eqref{eq: stability_error_bound} below, \Cref{lem: step_size_AoI_sum} will yield that $e_n \in o(1)$ almost surely. The convergence of $x_{n}$ will then follow from the stochastic approximation literature \cite[Section 2]{borkar2009stochastic}.

First, we have a bound for the drift errors using the Lipschitz-continuity of $h$ (\Cref{asm: lipschitz}):
\begin{align}
    \norm{e^i_n} &\le L \sum_{j=1}^D\norm{x^j_n - x^j_{n- \tau_{ij}(n)}}
\end{align}
Next, using triangular inequality it follows that
\begin{align}
    \norm{x^j_n - x^j_{n- \tau_{ij}(n)}}&\le  \sum_{k=n-\tau_{ij}(n)}^{n-1} \norm{x^j_{k+1} - x^j_{k}},\\
    &= \sum_{k=n-\tau_{ij}(n)}^{n-1} a(k) \left( \norm{h^j( x^1_{k-\tau_{j1}(k)}, \ldots, x^D_{k-\tau_{jD}(k)}) + M^j_k }  \right), \label{eq: stability_error_bound}
\end{align}
where the last step uses the main iteration \eqref{eq: main_iteration}.\footnote{Note that whenever $\tau_{ij}(n) = 0$ the sums on the right hand side (r.h.s.) are empty and thus equal to zero.} From inequality \eqref{eq: stability_error_bound}, one can suspect that $\sum_{k=n-\tau_{ij}(n)}^{n-1} a(k)  \to 0 \as $ is a sufficient condition to prove the stability of \eqref{eq: main_iteration}.  This is indeed immediate from the BMT provided that 1. the drift $h$ is bounded almost surely and 2. the noise $M_{n+1}$ is bounded almost surely. We do not make these assumptions.

We shall impose the natural condition that iteration \eqref{eq: main_iteration} (and thence \eqref{eq: main_iteration_with_error}) starts from a  prescribed $x_0$ with $(\Ew{\norm{x_0}^2})^\frac{1}{2} < \infty$.
\Cref{asm: lipschitz} then implies the linear growth of $h(\cdot)$:
\begin{equation}
    \label{eq: L1_bound_h}
    \norm{h(x)} \le K(1 + \norm{x})
\end{equation}
for all $x\in \R^d$ for some $K>0$ depending on $x_0$. To simplify the presentation, we assume without loss of generality that the same constant $K$ holds for both \eqref{eq: L1_bound_h} and the inequality in \Cref{asm: noise}. Equation \eqref{eq: stability_error_bound} therefore leads to
\begin{equation}
    \label{eq: stability_error_bound_new2}
    \norm{x^j_n - x^j_{n- \tau_{ij}(n)}} \le K \sum_{k=n-\tau_{ij}(n)}^{n-1} a(k) \left( 1+ \norm{x_k} +  \sum_{l=1}^D \norm{x_k^l -  x^l_{k-\tau_{jl}(k)}}\right)  + \norm{\sum_{k=n-\tau_{ij}(n)}^{n-1} a(k)  M^j_k}. 
\end{equation}
We can now make a few key observations from this inequality.

\emph{First}, in view of the traditional BMT, we will create a rescaled trajectory $\hat{x}(t)$ from $x_k$. Then we want to show that this rescaled trajectory is stable almost surely. With inequality \eqref{eq: stability_error_bound_new2}, we can now see that the rescaling sequence $\norm{x(T_m)}$ as used in the BMT does not work for a distributed BMT. The inequality shows that a bound for the local drift error $e_n$ (thence for $x_n$) depends on $x_k$ with $k \in \{n-1, \ldots n-\tau_{ij}(n) \}$. The problem is that these $x_k$ will be associated with different $T$-length intervals $[T_{m}, T_{m+1}]$ than $x_n$, whenever $\tau_{ij}(n)$ is large. Thus scaling both sides of \eqref{eq: stability_error_bound_new2} with the scaling factor associated with $x_n$ does not lead to variables on the r.h.s. of \eqref{eq: stability_error_bound_new2} that can be meaningfully related to the rescaled trajectory $\hat{x}(t)$. We instead propose that the interpolated trajectory is scaled over every $T$-length segment $[T_{m}, T_{m+1}]$ using the scaling sequence $s(m) \coloneqq \sup_{l\le m} \norm{x(T_l)}$. The crucial point of this construction is that the $s(m)$ is monotonically increasing, which solves the aforementioned problem of the original scaling sequence.

Using the new rescaling sequence $s(m)$, we define the rescaled trajectory $\hat{x}(t)$ as well as rescaled versions $\hat{e}_n$ and $\hat{M}_{n+1}$ of $e_n$ and $M_{n+1}$, respectively. 
The task is then to show that $\hat{x}(t)$ is stable almost surely, for which we proceed in the following steps:
\begin{enumerate}
    \item We prove an $L_2$ bound, $\sup_t \Ew{ \norm{\hat{x}(t)}^2} < \infty$ (\Cref{lem: main_l2_bound})
    \item We show that the accumulated rescaled noise iteration $\sum_{k=0}^{n-1} a(k) \hat{M}_{k+1}$ is convergent almost surely (\Cref{lem: martingale_convergence}).
    \item We show that $\sup_{t\ge 0} \norm{\hat{x}(t)} < \infty \as$ and $\norm{\hat{e}_n} \to 0\as   $ (\Cref{lem: main_L1_bound} and \Cref{cor: L1_error_to_zero}, respectively).
\end{enumerate}
While straightforward in the BMT, these steps are much more involved in the presence of stochastic drift errors due to AoI. \Cref{eq: stability_error_bound_new2} sheds light on these difficulties.

The \emph{second} observation from \eqref{eq: stability_error_bound_new2} is that we need to bound $\Ew{\sum_{k=n-\tau_{ij}(n)}^{n-1} \norm{x_k^l -  x^l_{k-\tau_{jl}(k)}}^2}$ to show that $\hat{x}(t)$ is bounded in $L_2$ . This requires that we take the expected value of a random number of random variables, which is generally difficult without additional assumptions (see e.g. Wald's lemma). To circumvent this, we use the almost sure convergence from \Cref{lem: step_size_AoI_sum}.
Specifically, with \Cref{lem: step_size_AoI_sum} it follows from Egorov's theorem that 
$\sum_{k=n-\tau_{ij}(n)}^{n-1} a(k)$ converges almost uniformly. We can therefore work with a deterministic upper bound $\Delta(n) \ge \tau_{ij}(n)$, such that  $\sum_{k=n-\Delta(n)}^{n-1} a(k)$ converges uniformly on an increasing sequence of probability subspaces. By construction, it will follow that 
\begin{equation}
    \label{eq: deteministic_bound}
    \Ew{\sum_{k=n-\tau_{ij}(n)}^{n-1} a(k) \norm{x_k^l -  x^l_{k-\tau_{jl}(k)}}^2 } \le \sum_{k=n-\Delta(n)}^{n-1} a(k) \Ew{\norm{x_k^l -  x^l_{k-\tau_{jl}(k)}}^2 }.
\end{equation}
and we can show the $L_2$ bound 1. on an increasing sequence of probability subspaces. With this, we then conclude that $\sum_{k=0}^{n-1} a(k) \hat{M}_{k+1}$ converges almost surely on the increasing sequence of probability subspaces. Then, 2. follows on the whole underlying probability space by continuity from below. 

The \emph{third} observation from \eqref{eq: stability_error_bound_new2} is that all local drift errors $e^i_n$ are interdependent, i.e. the bound of each $\norm{x^j_n - x^j_{n- \tau_{ij}(n)}}$ depends on all $\norm{x^l_k - x^l_{k- \tau_{jl}(n)}}$ for all $1\le l \le D$ and $k \in \{n-1, \ldots n-\tau_{ij}(n) \}.$
Here, our important observation is that if we sum both sides of \eqref{eq: stability_error_bound_new2} overall $1\le i,j \le D$ then a recursive inequality in the variable $\sum_{i,j=1}^D \norm{x^j_n - x^j_{n- \tau_{ij}(n)}}$ arises. Indeed, this recursive structure arises both in $L_2$ as well as is in norm, i.e. in step 1. and step 3. above. To evaluate the recursive structures, we combine \Cref{lem: step_size_AoI_sum} and a new Gronwall-type inequality (\Cref{lem: new_gronwall}).
Finally, by combining the evaluated recursive structure with the rescaled algorithm iteration, we will arrive at an iteration that shows $\sup_{t\ge 0} \norm{\hat{x}(t)} < \infty \as$.


With the established stability of the rescaled trajectory, we will then see that $\hat{x}(t)$ tracks solutions to ODEs with drift $h_{s(m)}(\cdot)$, with $h_c(\cdot)$ from \Cref{asm: BMT} and $s(m)$ the new rescaling sequence used to create $\hat{x}(t)$. It is now left to prove \Cref{thm: main_distributed} using that $h_\infty(\cdot)$ from \Cref{asm: BMT} is the drift of an asymptotically stable ODE.
As in the traditional BMT, we assume by contradiction that $x_k$ is unstable, which implies that $s(m)$ diverges to infinity. This now leads to a new contradiction: We will show that the rescaling factors have to eventually decrease once they exceed a certain threshold. This line of argument is new compared to the traditional BMT and was necessary since we defined the scaling sequence as monotonically increasing. Indeed, the line of argument in the traditional BMT is not applicable to monotonically increasing scaling sequences, but, as we illustrated above, the monotonically increasing scaling sequence is required to deal with the drift errors due to AoI. 
On the flip side, our new line of argument for the distributed BMT can be used to prove the traditional BMT using our monotonically increasing scaling sequence. This arguably leads to a simpler proof of the traditional BMT. We will now present the detailed stability analysis.

\section{Stability Analysis}
\label{sec: main_results}

Divide the time axis $[0,\infty)$ using the stepsize $a(n)$ as follows. Define time instants 
\begin{equation}
\label{eq: time_instants}
    t(0) \coloneqq 0, \quad  t(n) \coloneqq \sum_{i=1}^{n-1} a(i), \quad \text{ for all } n\ge1.
\end{equation}
Now define an interpolated trajectory $\overline{x}(t)$, by setting $\overline{x}(t(n)) \coloneqq x_{n}$, $n\ge0$ and define all other points $t \in [0,\infty)$ by linear interpolation.
Fix $T>0$, and define 
\begin{equation}
    T_0 = 0, \quad T_{m+1} \coloneqq \min \{t(n) : t(n) \ge T_m + T\}.
\end{equation}
Then $T_m = t(n(m)) = \sum_{i=0}^{n(m)-1} a(i)$ for an increasing sequence $n(m) \nearrow \infty$. 

Next, consider a rescaled/projected version of $\overline{x}(t)$ by defining the trajectory  
\begin{equation}
    \label{eq: projections}
    \hat{x}(t) \coloneqq \frac{\overline{x}(t)}{s(m)}, \qquad t \in [T_m,T_{m+1}),
\end{equation}
using the monotonically increasing scaling sequence
\begin{equation}
    \label{eq: scaling_seq}
    s(m) \coloneqq  \sup_{l \le m} \norm{\overline{x}(T_l)} \vee 1
\end{equation}
In addition, define the rescaled error
\begin{equation}
    \hat{e}_n \coloneqq  \frac{e_n}{s(m)}, \qquad n(m) \le n  < n(m+1),
\end{equation}
and the rescaled martingale noise 
\begin{equation}
    \hat{M}_{n+1} \coloneqq  \frac{M_{n+1}}{s(m)}, \qquad n(m) \le n < n(m+1).
\end{equation}
As outlined in the previous subsection, the first step is to show that $\hat{x}(t)$ is bounded in $L_2$ and to establish the convergence of the accumulated rescaled noise sequence 
\begin{equation}
    \label{eq: accumulated_noise}
    \hat{\zeta}_n \coloneqq \sum_{k=1}^{n-1} a(k) \hat{M}_{k+1}.
\end{equation}

\subsection{Convergence of the accumulated rescaled martingale noise}
\label{sec: martingale}

It will become useful to have a function $m(n)$ that selects for each discrete time $n\ge0$ the corresponding segment $[T_{m(n)}, T_{m(n)+1})$ in continuous time. In other words, $m(n)$ is the largest interval index $m$, such that $T_m \le  \sum_{i=0}^{n-1} a(i)$.

\Cref{asm: noise} now leads
\begin{align}
    \Ew{\norm{\hat{M}^i_{n+1}}^2 \mid \cF_n} &= \Ew{\frac{\norm{M^i_{n+1}}^2}{s(m(n))^2} \mid \cF_n}
    = \frac{\Ew{\norm{ M^i_{n+1}}^2 \mid \cF_n}}{s(m(n))^2} \\ &\le \frac{\Ew{\norm{ M_{n+1}}^2 \mid \cF_n}}{s(m(n))^2} \\ &\le \frac{K^2\left(1+ \norm{(x^1_{n- \tau_{i1}(n)}, \ldots, x^D_{n-\tau_{iD}(n)})}^2 \right)}{s(m(n))^2} \\
    &\le \frac{K^2\left(1+ \norm{x_n}^2 + \sum_{j=1}^D\norm{x^j_n - x^j_{n- \tau_{ij}(n)}}^2 \right)}{s(m(n))^2} \\
    &\le K^2\left(1+ \norm{\hat{x}(t(n)}^2 + \sum_{j=1}^D \left( \frac{\norm{x^j_n - x^j_{n- \tau_{ij}(n)}} }{s(m(n))} \right)^2 \right)
\end{align}
Notably, the second equality uses that $s(m(n))$ is $\cF_n$ measurable with $\cF_n$ as defined in \Cref{asm: noise}. By taking the expected value and the square root of the last inequality, we arrive at
\begin{align}
    \Ew{\norm{\hat{M}^i_{n+1}}^2 }^{\frac{1}{2}} \le 
    K\left(1+ \Ew{\norm{\hat{x}(t(n)}^2}^{\frac{1}{2}} + \sum_{j=1}^D \Ew{\left( \frac{\norm{x^j_n - x^j_{n- \tau_{ij}(n)}} }{s(m(n))} \right)^2}^{\frac{1}{2}} \right)
    \label{eq: L2_bound_M}
\end{align}
for all $n\ge 1$.  Further, using the Lipschitz-continuity of $h$, we can bound the rescaled additive drift errors \eqref{eq: local_drift_errors} in $L_2$ by
\begin{equation}
    \label{eq: init_error_bound}
    \Ew{\norm{\hat{e}^i_n}^2}^\frac{1}{2} \le L \sum_{j=1}^D\Ew{ \left(\frac{  \norm{x^j_n - x^j_{n- \tau_{ij}(n)}}}{s(m(n)} \right)^2}^{\frac{1}{2} }.
\end{equation}

Now divide both sides of the rewritten main iteration \eqref{eq: main_iteration_with_error} by $s(m(n))$, take the norm on both sides, and use  \eqref{eq: L1_bound_h}. Then,
\begin{equation}
    \norm{\hat{x}(t(n+1))} \le \norm{\hat{x}(t(n))} (1+a(n)K) + a(n)(1+ \norm{\hat{e}_n} + \norm{\hat{M}_{n+1}}).
\end{equation}
Finally, take $\Ew{(\cdot)^2}^\frac{1}{2}$ on both sides above and use \eqref{eq: L2_bound_M} and \eqref{eq: init_error_bound} componentwise to arrive at the following recursive $L_2$ bound for $\norm{\hat{x}(t(n))}$:
 \begin{align}
 \label{lem: distrete_L2_bound}
     \Ew{\norm{\hat{x}(t(n+1))}^2}^\frac{1}{2}  &\le \Ew{\norm{\hat{x}(t(n))}^2}^\frac{1}{2} (1+a(n)K_1) \nonumber \\ & \quad +  a(n)\left(K_1 + K_2 \Ew{\left( \frac{\norm{x^j_n - x^j_{n- \tau_{ij}(n)}} }{s(m(n))} \right)^2}^{\frac{1}{2}}\right).
 \end{align}
 with $K_1 \coloneqq  K(D+1) $ and $K_2 \coloneqq L + K$.

Next consider the local errors $x^j_n - x^j_{n- \tau_{ij}(n)}$ due to AoI. \Cref{eq: stability_error_bound_new2} leads to
\begin{align}
      \norm{x^j_n - x^j_{n- \tau_{ij}(n)}} \le K \sum_{k=n-\tau_{ij}(n)}^{n-1} a(k) \left( 1+ \norm{x_k} +  \sum_{l=1}^D \norm{x_k^l -  x^l_{k-\tau_{jl}(k)}}\right)  + \sum_{k=n-\tau_{ij}(n)}^{n-1} a(k)  \norm{M^j_k}. \label{eq: local_error_prelim_ineq_1}
\end{align}
To move forward, we have to take the expected value of \eqref{eq: local_error_prelim_ineq_1}. As discussed in the previous subsection, this is an expected value of a random number of random variables. \Cref{lem: step_size_AoI_sum} and Egorov's theorem now imply that $ \sum_{k=n-\tau_{ij}(n)}^{n-1} a(k)$ converges almost uniformly. We can therefore work with deterministic upper bounds for all $\tau_{ij}(n)$ on an increasing sequence of probability subspaces.

\begin{definition}[Almost uniform convergence]
Let $X_n, X$ be random variables on a probability space $(\Omega, \cF, \P)$.
Then $X_n$ are said to converge to $X$ almost uniformly if, for every $\varepsilon>0$, there exists an exceptional set $A \in \cF$ with $\Pr{A} < \varepsilon$ such that $X_n$ converges uniformly to X on the complement $E = \Omega \setminus A$.
\end{definition}
\noindent \textbf{Egorov's theorem}. Let $X_n, X$ be random variables. Then $X_n \to X$ almost surely if and only if $X_n \to X$ almost uniformly.

Let $(\Omega,\cF, \P)$ be the underlying probability space, i.e. $(\Omega,\cF, \P)$ is the common probability space on which the stochastic processes $\{x_n\}_{n\ge1}$, $\{M_{n+1}\}_{n\ge1}$ and  all $\{\tau_{ij}(n)\}_{n\ge1}$ are defined. Fix a sequence
$\{\varepsilon_z\}_{z\ge0} \subset (0,1)$ with $\varepsilon_z \to 0$. \Cref{lem: step_size_AoI_sum} shows that all
$\sum_{k=n-\tau_{ij}(n)}^{n-1} a(k) \to 0$ almost surely. 
Egorov's theorem thus implies that there are sets $E_0, E_1, \ldots, E_z, \ldots \in \cF$ with $\Pr{E_z} \ge 1- \varepsilon_z$, such that $\sum_{k=n-\tau_{ij}(n)}^{n-1} a(k) \to 0$ uniformly on each $E_z$. We may assume that $E_0 \subset E_1 \subset E_2 \ldots $ as uniform convergence on finite unions follows from the uniform convergence on the individual sets. 

Recall that the objective of this subsection is to show that $\Pr{ \hat{\zeta}_n \text{ converges} } = 1$.
Suppose that we can show that  $\hat{\zeta}_n$ converges on any restricted probability space $(E_z, \cF_z, \P_z)$ with sub-sigma algebra $\cF_z \coloneqq \{E_z \cap E \mid E\in \cF \}$ and restricted probability measure $\P_z \coloneqq \frac{\P \vert_{\cF_z}}{\Pr{E_z}}$
for all $z\ge 0$, i.e., suppose we can show that 
\begin{equation}
    \P_z \left(\hat{\zeta}_n \text{ converges} \right) = \frac{\P\left( E_z  \cap \{   \hat{\zeta}_n \text{ converges} \} \right)}{ \Pr{E_z}} = 1.
\end{equation}
Then, by the construction of $E_z$, we have
\begin{equation}
    \P\left( E_z  \cap \{  \hat{\zeta}_n \text{ converges} \} \right) = \Pr{E_z} \ge 1 - \varepsilon_z.
\end{equation}
Using continuity from below, as $E_z$ are increasing, it follows that
\begin{equation}
    \P\left( \bigcup_{z\ge 0} E_z  \cap \{ \hat{\zeta}_n \text{ converges} \} \right) = 1,
\end{equation}
since $\varepsilon_z \to 0$ as $z \to \infty$. Thus $\Pr{ \hat{\zeta}_n \text{ converges}} = 1$.

From the hitherto presented discussion, it follows that we can assume without loss of generality that all $\sum_{k=n-\tau_{ij}(n)}^{n-1} a(k) \to 0$ uniformly on $\Omega$. If this is not true, then the rest of this subsection shows  that $\hat{\zeta}_n$ converges on each restricted probability space $(E_z, \cF_z, \P_z)$ using that $\sum_{k=n-\tau_{ij}(n)}^{n-1} a(k) \to 0$ uniformly on each restricted probability space. Then the convergence of $\hat{\zeta}_n$ follows from the line of argument in the previous paragraph.
 
Consider the deterministic sequence
\begin{equation}
    \Delta(n) \coloneqq \sup_{\omega \in \Omega} \{\tau_{ij}(n)(\omega) \mid 1\le i,j \le D \}.
\end{equation}
The assumed uniform convergence thus implies that $\sum_{k=n-\Delta(n)}^{n-1} a(k) \to 0$ as $n \to \infty$.
Further, as $\tau_{ij}(n) \le \Delta(n)$ for all $1\le i,j \le D$ and all $n\ge 1$, \Cref{eq: local_error_prelim_ineq_1} leads to
\begin{align}
      \norm{x^j_n - x^j_{n- \tau_{ij}(n)}} \le K \sum_{k=n-\Delta(n)}^{n-1} a(k) \left( 1+ \norm{x_k} +  \sum_{l=1}^D \norm{x_k^l -  x^l_{k-\tau_{jl}(k)}}\right)  + \sum_{k=n-\Delta(n)}^{n-1} a(k)  \norm{M^j_k}. 
\end{align}
Next, divide the above inequality by $s(m(n))$ and use that $s(m(n))$ is by construction monotonically increasing. It follows that
\begin{align}
      \frac{\norm{x^j_n - x^j_{n- \tau_{ij}(n)}}}{s(m(n))} \le K \sum_{k=n-\Delta(n)}^{n-1} a(k) \left( 1+ \norm{\hat{x}_k} +  \sum_{l=1}^D \frac{\norm{x_k^l -  x^l_{k-\tau_{jl}(k)}}}{s(m(k))}\right)  + \sum_{k=n-\Delta(n)}^{n-1} a(k)  \norm{\hat{M}^j_k}. 
\end{align}
Since $\Delta(n)$ is deterministic, we can evaluate $\Ew{(\cdot)^2}^{\frac{1}{2}}$ on both sides above and apply \eqref{eq: L2_bound_M}, then
\begin{align}
    \Ew{\left( \frac{\norm{x^j_n - x^j_{n- \tau_{ij}(n)}}}{s(m(n))}\right)^2}^{\frac{1}{2}} &\le 2K \sum_{k=n-\Delta(n)}^{n-1} a(k) (1 + \Ew{\norm{ \hat{x}(t(k))}^2}^{\frac{1}{2}} )\nonumber \\
    & \qquad + 2K \sum_{k=n-\Delta(n)}^{n-1} a(k) \left( \sum_{l=1}^D \Ew{\left(\frac{\norm{x^l_k - x^l_{k- \tau_{jl}(k)}}}{s(m(k))}\right)^2}^{\frac{1}{2}} \right).
\end{align}
Finally, a summation over all $1\le i,j \le D$ leads to
\begin{align}
    \label{eq: error_iterative}
    \sum_{i,j = 1}^D \Ew{\left( \frac{\norm{x^j_n - x^j_{n- \tau_{ij}(n)}}}{s(m(n))}\right)^2}^{\frac{1}{2}} &\le  2K D^2 \sum_{k=n-\Delta(n)}^{n-1} a(k) \left(1 + \Ew{\norm{ \hat{x}(t(k))}^2}^{\frac{1}{2}} \right) \nonumber \\
    &+ 2KD \sum_{k=n-\Delta(n)}^{n-1} a(k) \sum_{i,j=1}^D \Ew{\left( \frac{\norm{x^j_k - x^j_{k- \tau_{ij}(k)}}}{s(m(k))}\right)^2}^{\frac{1}{2}}.
\end{align}
This is the announced recursive inequality of the AoI error in $L_2$. To evaluate this inequality, we propose a new backward Gronwall-type inequality with a varying lower time horizon. 
The main observation is that the sum $\sum_{k=n-\Delta(n)}^{n-1} a(k)$ leads to some weighted averaging, such that the left hand of \eqref{eq: error_iterative} can be bounded by a function of $\Ew{\norm{ \hat{x}(t(k))}^2}^{\frac{1}{2}}$, $k\le n-1$.

We give the new Gronwall-type inequality in greater generality than necessary for our analysis, since it may be of independent interest.  We present the proof in Appendix \ref{sec: appendix}.
\begin{lemma}
\label{lem: new_gronwall}
    Let $\{y_n\}$, $\{a_n\}$, $\{b_n\}$, $\{\Delta_n\}$ be non-negative sequences, $\{c_n\}$ be a  monotonically increasing non-negative sequence and $C,B >0$ be scalars, such that for all $n$,
    \begin{equation}
        y_n \le b_nc_n + C \sum_{k=n - \Delta_n}^{n-1} a_k y_k, \qquad   \sup_{n\ge 0} b_n \le B,
    \end{equation}
    \begin{equation}
        \label{eq: averaging_condition}
        N \coloneqq \inf \{ N \in \N : C \sum_{k=n - \Delta_n}^{n-1} a_k \le \frac{e^{Ct(N)} - 1}{e^{Ct(N)}} \text{ for all } n \ge N\} < \infty,
    \end{equation}
    with $t(n)$ as defined in \eqref{eq: time_instants}.
    Then
    \begin{equation}
        y_n \le c_n \left( b_n + B C e^{C t(N)} \left(\sum_{k=n - \Delta_n}^{n-1} a_k \right) \right).
    \end{equation}
\end{lemma}
\begin{corollary}
    Consider the setting in \Cref{lem: new_gronwall} with $b_n \in o(1)$ and $\sum_{k=n - \Delta_n}^{n-1} a_k \to 0$, then $y_n \in o(c_n)$.
\end{corollary}

\begin{lemma}
\label{lem: error_L2_bound}
There is constant $K_3>0$, such that
\begin{equation}
    \sum_{i,j = 1}^D \Ew{\left( \frac{\norm{x^j_n - x^j_{n- \tau_{ij}(n)}}}{s(m(n))}\right)^2}^{\frac{1}{2}}
    \le K_3 \left(1 + \sup_{k\le n-1 } \Ew{\norm{\hat{x}(t(k))}^2}^{\frac{1}{2}} \right) \left(\sum_{k=n-\Delta(n)}^{n-1} a(k) \right)
\end{equation}
\end{lemma}
\begin{proof}
    Equation \eqref{eq: error_iterative} leads to
    \begin{align}
    \label{eq: error_iterative2}
    \sum_{i,j = 1}^D \Ew{\left( \frac{\norm{x^j_n - x^j_{n- \tau_{ij}(n)}}}{s(m(n))}\right)^2}^{\frac{1}{2}} &\le  2K D^2 \left( \sum_{k=n-\Delta(n)}^{n-1} a(k) \right) \left(1 + \sup_{k\le n-1 } \Ew{\norm{ \hat{x}(t(k))}^2}^{\frac{1}{2}} \right) \nonumber \\
    &+ 2KD \sum_{k=n-\Delta(n)}^{n-1} a(k) \sum_{i,j=1}^D \Ew{\left( \frac{\norm{x^j_k - x^j_{k- \tau_{ij}(k)}}}{s(m(n))}\right)^2}^{\frac{1}{2}}.
\end{align}
    Now define $y_n \coloneqq \sum_{i,j = 1}^D \Ew{\left( \frac{\norm{x^j_n - x^j_{n- \tau_{ij}(n)}}}{s(m(n))}\right)^2}^{\frac{1}{2}}$,
    $b_n \coloneqq 2K D^2 \left( \sum_{k=n-\Delta(n)}^{n-1} a(k) \right)$, \\$c_n \coloneqq 1 + \sup_{k\le n-1 } \Ew{\norm{ \hat{x}(t(k))}^2}^{\frac{1}{2}}$, $C \coloneqq 2KD$ as well as  $a_n$ and $\Delta_n$ as evident. The lemma now immediately follows from our new Gronwall-type inequality.
\end{proof}

We are now ready to show that $\norm{\hat{x}(t)}$ is bounded in $L_2$. 

\begin{lemma}
\label{lem: main_l2_bound}
\begin{equation}
     \sup_{t\ge 0} \Ew{\norm{\hat{x}(t)}^2} < \infty 
\end{equation}
\end{lemma}
\begin{proof}
Fix $m\ge 0$ and $n(m) \le n < n(m+1)$. Insert the inequality from \Cref{lem: error_L2_bound} into \eqref{lem: distrete_L2_bound}, then 
\begin{align}
   \Ew{\norm{\hat{x}(t(n+1))}^2 }^{\frac{1}{2}}&\le  \Ew{\norm{\hat{x}(t(n))}^2 }^{\frac{1}{2}} (1+a(n)K_1) \\ &+  a(n) \left(K_1 + K_2 K_3\left(1+ \sup_{k\le n-1 } \Ew{\norm{ \hat{x}(t(k))}^2}^{\frac{1}{2}}\right) \left(\sum_{k=n-\Delta(n)}^{n-1} a(k) \right) 
   \right). \nonumber 
\end{align}
Keeping in mind that 
\begin{equation}
    \label{eq: sum_bound_T}
    \sum_{n=n(m)}^{n(m+1)-1} a(n) \le T+1, \quad \norm{\hat{x}(t(n(m)))} \le 1
\end{equation}
a simple recursion then shows that 
\begin{align}
    \label{eq: iteration_L2_bound}
    &\Ew{\norm{\hat{x}(t(n+1))}^2 }^{\frac{1}{2}} \le \exp({K_1(T+1)})( 1 +K_1(T+1) )\\ & \qquad + \exp({K_1(T+1)})K_2 K_3  \left( 1 + \sup_{k\le n-1 } \Ew{\norm{ \hat{x}(t(k))}^2}^{\frac{1}{2}}\right) \left(\sum_{k= n(m)}^n a(k)  \left(\sum_{l=k-\Delta(k)}^{k-1} a(l) \right)   \right), \nonumber
\end{align}
where we used that $1+a(n)K_1 \le \exp(a(n)K_1)$. As $\Ew{\norm{ \hat{x}(t(1))}^2}^{\frac{1}{2}} < \infty$, it follows from
\eqref{eq: iteration_L2_bound} that $\Ew{\norm{ \hat{x}(t(n))}^2}^{\frac{1}{2}} < \infty$ for every fixed $n\ge 0$. Furthermore, it follows from \Cref{lem: step_size_AoI_sum} and \eqref{eq: sum_bound_T} that $\sum_{k= n(m)}^n a(k)  \left(\sum_{l=k-\Delta(k)}^{k-1} a(l) \right) \in o(1)$.  The statement of the lemma is now immediate from these two conclusions and \eqref{eq: iteration_L2_bound}. A formal proof can be given by induction.
\end{proof}

We are now ready to prove the convergence of the accumulated rescaled noise iteration. 
\begin{lemma}
\label{lem: martingale_convergence}
    $\hat{\zeta}_n$ converges almost surely.
\end{lemma}
\begin{proof}
By the convergence theorem for square-integrable martingales \cite[Appendix C, Theorem 11]{borkar2009stochastic}, it is enough to show that
$\sum_{n\ge1} \Ew{\norm{a(n) \hat{M}_{n+1}}^2 \mid \cF_n } < \infty \as$.
We have
\begin{align}
    &\Ew{\sum_{n\ge1} a(n)^2 \Ew{\norm{\hat{M}_{n+1}}^2 \mid \cF_n }} = \sum_{n\ge1} a(n)^2 \Ew{\norm{\hat{M}_{n+1}}^2 } \\
    &\le \sum_{n\ge1} a(n)^2 \left( KD\left(1 + \Ew{\norm{\hat{x}(t(n))^2 }}^{\frac{1}{2}} \right) 
    + K \sum_{i,j= 1}^D\Ew{\left(\frac{\norm{x^j_n - x^j_{n- \tau_{ij}(n)}}}{s(m(n))}\right)^2}^{\frac{1}{2}} \right)^2 
    \label{eq: acc_noise_bound},
\end{align}
using \eqref{eq: L2_bound_M}. \Cref{lem: error_L2_bound},
\Cref{lem: main_l2_bound} and the square summability of $a(n)$ therefore imply that \eqref{eq: acc_noise_bound} is finite and the statement follows.
\end{proof}
A corollary to \Cref{lem: martingale_convergence} will become useful.
\begin{corollary}
\label{cor: noise_L1_error}
    $\frac{1}{s(m(n))} \sum_{k= n-\tau_{ij}(n)}^{n-1} a(k) M^j_{k+1} \to 0 \as$
\end{corollary}
\begin{proof}
    We have that
    \begin{equation}
        \frac{1}{s(m(n))} \sum_{k= 1}^{n} a(k) M^j_{k+1}
        =  \frac{1}{s(m(n))} \sum_{k= 1}^{n}s(m(k)) a(k) \hat{M}^j_{k+1}.
    \end{equation}
    Now for every sample point, there are two scenarios: 
    \begin{enumerate}
        \item $s(m(n))$ is bounded, then $\sum_{k= 1}^{n}s(m(k)) a(k) \hat{M}^j_{k+1}$ converges by Abel's test for infinite series.
        \item $s(m(n))$ is unbounded, then $\frac{1}{s(m(n))} \sum_{k= 1}^{n}s(m(k)) a(k) \hat{M}^j_{k+1} \to 0$ by Kronecker's lemma.
    \end{enumerate}
    It follows that $\frac{1}{s(m(n))} \sum_{k= 1}^{n} a(k) M^j_{k+1}$ converges almost surely. As all $n-\tau_{ij}(n) \to \infty \as$, the lemma follows.
\end{proof}

\subsection{Stability of the rescaled trajectory}

We are now ready to show that $\hat{x}(t)$ is stable.
Recall the functions $h_c(x)$, $c\in [1,\infty]$,  defined in (A5). It is not difficult to verify that:
\begin{enumerate}
    \item $h_c$, $h_\infty$ are Lipschitz-continuous with the same Lipschitz constant $L>0$ as $h$.
    \item $\norm{h_c(x)} \le K (1 + \norm{x})$ for every $c \in [1,\infty]$.
\end{enumerate}

\begin{lemma}
\label{lem: main_L1_bound}
$\sup_{t\ge 0} \norm{\hat{x}(t)} < \infty \as.$
\end{lemma}
\begin{proof}
Fix $m>0$ and $n(m) \le n < n(m+1)$. The rescaled iteration can be written as
\begin{equation}
    \label{eq: projected_iteration}
    \hat{x}(t(n+1)) = \hat{x}(t(n)) + a(n) \left[ h_{s(m)}(\hat{x}(t(n))) + \hat{e}_n + \hat{M}_{n+1} \right],
\end{equation}
using the functions $h_c(\cdot)$ from (A5). Moreover, 
\begin{equation}
    \label{eq: projected_iteration_bound_1}
    \norm{\hat{e}_n} \le L \sum_{i,j=1}^D  \frac{\norm{x_n^j - x^j_{n-\tau_{ij}(n)}} }{s(m)}
\end{equation}
We will now define a point-wise upper bound for all $\tau_{ij}(n)$. Define
\begin{equation}
    \tau(n) \coloneqq \max \{\tau_{ij}(n) \mid 1 \le i,j \le D \}.
\end{equation}
Notice that $\tau(n)$ is a stochastic bound in contrast to the deterministic bound used in \Cref{lem: main_l2_bound}. 
As the number of iterations is finite ($D < \infty$), it follows from \Cref{lem: step_size_AoI_sum} that $\sum_{k= n-\tau(n)}^{n-1} a(k) \to 0 \as$.
Starting from \eqref{eq: stability_error_bound_new2}, we now obtain that
\begin{align}
    \label{eq: L1_error_iterative}
    \sum_{i,j=1}^D \frac{\norm{x_n^j - x^j_{n-\tau_{ij}(n)}}}{s(m(n))} &\le KD^2 \left(\sum_{k= n-\tau(n)}^{n-1} a(k) \right) (1+ \sup_{k \le n-1} \norm{\hat{x}(t(k))})\nonumber
    \\ &+ K D \sum_{k= n-\tau(n)}^{n-1} a(k) \sum_{i,j=1}^D \frac{\norm{x_k^j - x^j_{k-\tau_{ij}(k)}}}{s(m(k))}  \nonumber \\  &+ \sum_{i,j=1}^D \frac{1}{s(m(n))}\norm{\sum_{k= n-\tau_{ij}(n)}^{n-1} a(k) M^j_{k+1}},
\end{align}
where we again used the monotonicity of $s(m(n))$. Please notice the similarity to \eqref{eq: error_iterative2}. Further, the third term in \eqref{eq: L1_error_iterative} converges to zero by \Cref{cor: noise_L1_error}. Using the new Gronwall-type inequality \Cref{lem: new_gronwall} it follows that there is a constant $K_4>0$, such that
\begin{align}
\label{eq: projected_iteration_bound_2}
    &\sum_{i,j=1}^D \frac{\norm{x_n^j - x^j_{n-\tau_{ij}(n)}}}{s(m(n))} \le K_4\left( 1+ \sup_{k \le n-1} \norm{\hat{x}(t(k))}\right) \left(\sum_{k= n-\tau(n)}^{n-1} a(k) \right)
\end{align}

Iterating the rescaled iteration \eqref{eq: projected_iteration} now yields, for $ 0 < k \le n(m+1) - n(m)$,
\begin{equation}
    \begin{split}
        \hat{x}(t(n(m)+k)) &= \hat{x}(t(n(m))) + \sum_{i=0}^{k-1} a(n(m) + i)  h_{s(m)}(\hat{x}(t(n(m)+i))) \\ & + \sum_{i=0}^{k-1} a(n(m) + i) \hat{e}_{n(m)+i}  + \left(\hat{\zeta}_{n(m)+k} - \hat{\zeta}_{n(m)} \right),
    \end{split}
\end{equation}
with
\begin{equation}
    \norm{h_{s(m)}(\hat{x}(t(n(m)+i)))} \le K\left( 1 + \norm{\hat{x}(t(n(m)+i))} \right).
\end{equation}
Further, \Cref{lem: martingale_convergence} implies that $\sup_{n\ge1} \norm{\hat{\zeta}_n} < \infty \as.$ Hence, there is sample path dependent constant $B>0$ such that $\sup_{n\ge1} \norm{\hat{\zeta}_n} \le B$.
Then, since $\sum_{0 \le i < n(m+1) - n(m)} a(n(m) + i) \le T+1$ and $\norm{\hat{x}(t(n(m)))} \le 1$, it  follows that
\begin{equation}
    \begin{split}
        \norm{\hat{x}(t(n(m)+k))} &\le 1 + K(T+1)  + K\sum_{i=0}^{k-1} a(n(m) + i) \norm{\hat{x}(t(n(m)+i))} \\
        &+ \sum_{i=0}^{k-1} a(n(m) + i) \norm{\hat{e}_{n(m)+i}} + 2B
    \end{split}
\end{equation}
The traditional discrete Gronwall inequality, \Cref{gronwall_ineq}, now shows that
\begin{equation}
    \label{eq: L1_gronwall_bound}
    \norm{\hat{x}(t(n(m)+k))} \le \left( 1 + K(T+1) + 2B + \sum_{i=0}^{k-1} a(n(m) + i) \norm{\hat{e}_{n(m)+i}}\right) \exp(K(T+1))
\end{equation}
for $ 0 < k \le n(m+1) - n(m)$. The combination of \eqref{eq: projected_iteration_bound_1}, \eqref{eq: projected_iteration_bound_2} and \eqref{eq: L1_gronwall_bound} then yields that
\begin{align}
    &\norm{\hat{x}(t(n(m)+k))} \le \left( 1 + K(T+1) + 2B \right) \exp(K(T+1)) \\
    & + \exp(K(T+1)) \sum_{i=0}^{k-1} a(n(m) + i) L K_4 \left( 1+ \sup_{k \le n(m)+i-1} \norm{x_k} \right)\left(\sum_{k= n(m)+i-\tau(n(m)+i)}^{n(m)+i-1} a(k) \right)  \nonumber
\end{align}
\Cref{lem: step_size_AoI_sum} therefore implies that 
\begin{equation}
    \norm{\hat{x}(t(n)} \le K_5 + \frac{1}{2} \sup_{k < n} \norm{\hat{x}(t(k))}  
\end{equation}
for large $n$ and some $K_5>0$. The lemma now follows. As in \Cref{lem: main_l2_bound}, a formal proof can be given by induction.
\end{proof}
\begin{corollary}
\label{cor: L1_error_to_zero}
$\norm{\hat{e}_n} \to 0 \as$
\end{corollary}
\begin{proof}
    Combine \eqref{eq: projected_iteration_bound_1} and \eqref{eq: projected_iteration_bound_2}, then apply  \Cref{lem: step_size_AoI_sum} and \Cref{lem: main_L1_bound}.
\end{proof}

\subsection{Stability Theorem}

We can now conclude that the rescaled trajectory is a noisy approximation of solutions to ODEs with drift $h_{s(m)}(\cdot)$.
For any $m\ge0 $, let $x^m(t)$, $t\in [T_m,T_{m+1}]$, be the unique solution to the ODE
\begin{equation}
    \label{eq: scaledODE}
    \dot{x}(t) = h_{s(m)} (x(t))
\end{equation}
with initial condition $x^m(T_m) = \hat{x}(T_m)$.
Recall again that the rescaled iteration can be written as
\begin{equation}
\hat{x}(t(n+1)) = \hat{x}(t(n)) + a(n) \left[ h_{s(m(n))}(\hat{x}(t(n))) + \hat{e}_n + \hat{M}_{n+1} \right],
\end{equation}
for every $n \ge 1$.
\Cref{lem: martingale_convergence} and \Cref{cor: L1_error_to_zero} imply that the rescaled iteration $\hat{x}(t(n))$ has the form of a standard stochastic approximation iteration with convergent accumulated noise $\hat{\zeta}_n$ and vanishing additive error $\hat{e}_n$, respectively. Further, \Cref{lem: main_L1_bound} shows that the projected iteration remains bounded almost surely. 
A standard stochastic approximation argument, 
then shows that $\hat{x}(t)$ is a noisy approximation of solutions to the ODEs \eqref{eq: scaledODE}.
\begin{lemma}
\label{lem: SA_lemma}
$$\lim\limits_{m\to \infty} \sup\limits_{t\in [T_m,T_{m+1}]} \norm{\hat{x}(t) - x^m(t)} = 0 \as$$ 
\end{lemma}

Assumption (A5) ensures the existence of a function $h_\infty(x)$, such that $h_c(x) \to h_\infty(x)$ uniformly on compact sets as $c\to\infty$. Further, $h_\infty$ has the origin as its unique globally asymptotically stable equilibrium. Therefore solutions to $\dot{x}(t) = h_\infty(x)$ will eventually reach a neighborhood of the origin after some $T >0$.  

Recall that the functions $h_c$ inherit the Lipschitz-continuity from $h.$ Thus the scaled ODEs have unique solutions for every initialization. For every $c \in (0,\infty]$, let $\phi_c(t,x)$ denote the unique solution of the ODE $\dot{x}(t) = h_c(x(t))$ with initial condition $x$.
The compact convergence of $h_c \to h_\infty$ now guarantees that for large scaling factors $c$ and initialization on the unit ball, the ODE solutions $\phi_c(t,x)$ will reach a neighborhood of the closed unit ball after $T>0$. This is stated as the following lemma; we refer to \cite[Chapter 3]{borkar2009stochastic} where something very similar has been shown.

\begin{lemma}
\label{lem: liapunov_arg}
There exist $c_0 > 0$ and $T>0$ such that for all initial conditions x on the closed unit ball, $\norm{\phi_c(x,t)} < \frac{1}{2}$ for $t\in [T,T+1]$ and $c>c_0$.
\end{lemma}

We are now ready to prove our main stability theorem.
\begin{proof}[Proof of \Cref{thm: main_distributed}]
Fix $T>0$ from \Cref{lem: liapunov_arg} and for this $T$ a sample point where \Cref{lem: martingale_convergence} and \Cref{lem: SA_lemma} hold. We present a proof by
contradiction. Recall the scaling sequence $s(m) = \sup_{l \le m} \norm{\overline{x}(T_l)} \vee 1$ and suppose that $\sup\limits_{m\ge0} \norm{\overline{x}(T_m)} < \infty $ does not hold. Then, by construction, $s(m) \nearrow \infty$ monotonically. 

We will now consider those time intervals where the scaling sequence equals the norm of the trajectory a the beginning of the corresponding intervals, i.e. those $m$ where $s(m) = \norm{\overline{x}(T_m)}$. That is, we consider those time steps where $\norm{\hat{x}(T_{m})} = 1$. This yields a subsequence of interval indices $\{\tilde{m}(k)\}_{k\ge 0} \subset \{m\}_{m\ge0}$, such that $\norm{\hat{x}(T_{\tilde{m}(k)})} = 1$ and, for all $k\ge0$, and
\begin{equation}
    s(l) = \overline{x}(T_{\tilde{m}(k)})
\end{equation}
for all $l \in \{\tilde{m}(k), \ldots, \tilde{m}(k+1)-1\}$. In other words, $\tilde{m}(k)$ is the sequence of time steps that defines the rescaling sequence $s(m)$.


Recall now that $x^m(t)$, $t\in [T_m,T_{m+1}]$, are the unique solutions to $\dot{x}(t) = h_{s(m)} (x(t))$ with initial condition $x^m(T_m) = \hat{x}(T_m) = \frac{\overline{x}(T_m)}{s(m)}$. Since $s(m) \nearrow \infty$, there will be some $k'$, such that $s(m(k)) > c_0$ with $c_0$ from \Cref{lem: liapunov_arg} for all $k \ge k'$.

For these $k\ge k'$ we will now consider the last interval, $\tilde{m}(k+1)-1$, from the set of intervals $\{\tilde{m}(k), \ldots, \tilde{m}(k+1)-1 \} $ where $s(\tilde{m}(k))$ is used as the scaling sequence.
In particular, \Cref{lem: liapunov_arg} shows that
\begin{equation}
    \norm{x^{\tilde{m}(k+1)-1}(T_{\tilde{m}(k+1)})} < \frac{1}{2}.
\end{equation}
for $k\ge k'$. 
Furthermore, \Cref{lem: SA_lemma} shows that there is some $k''$, such that
\begin{equation}
    \sup\limits_{t\in [T_{\tilde{m}(k)},T_{\tilde{m}(k+1)}]} \norm{\hat{x}(t) - x^{m(k)}(t)} < \frac{1}{2}, \quad  k\ge k''.
\end{equation}
We can finally conclude that
\begin{align} \frac{s(\tilde{m}(k+1))}{s(\tilde{m}(k))} = \frac{\overline{x}(T_{\tilde{m}(k+1)})}{\overline{x}(T_{\tilde{m}(k)})} &= \lim\limits_{t \nearrow T_{\tilde{m}(k+1)}} \norm{\hat{x}(t)} \\
&\le \lim\limits_{t \nearrow T_{\tilde{m}(k+1)}} \norm{x^{\tilde{m}(k+1)-1}(t)} + \lim\limits_{t \nearrow T_{\tilde{m}(k+1)}} \norm{\hat{x}(t) -x^{\tilde{m}(k+1)-1}(t)} \\
&< \frac{1}{2} + \frac{1}{2} = 1
\end{align}
for all $k>\max(k',k'')$. This is the required contradiction since $s(m)$ was constructed as monotonically increasing.
Hence, $\sup_{m\ge0} s(m) < \infty $ almost surely and the theorem follows form \Cref{lem: main_L1_bound}.
\end{proof}
\begin{proof}[Proof of \Cref{cor: convergence}]
    Under (A1)-(A5) it follows from \Cref{thm: main_distributed} that $x_n$ is almost surely stable. Using \Cref{lem: SA_lemma} it then follows that $x_{n+1} = x_n + a(n) [h(x_n) + e_n + M_{n+1}]$ with $e_n \in o(1)$. As mentioned in the introduction, the convergence now follows from the SA literature, see e.g. \cite[Sec. 2.2]{borkar2009stochastic}.
\end{proof}

\subsection{A generalization of the Borkar-Meyn Theorem}
\label{sec: genBMT}

Inspired by our analysis of the distributed BMT using a monotonically increasing scaling sequence, we can now propose a weaker version of \Cref{asm: BMT}:
\begin{assumption}
    \label{asm: BMT_new}
    The functions $h_c(x) \coloneqq \frac{h(cx)}{c}, c\ge 1$, $x\in \R^d$, satisfy that there exists a sequence $c_n \nearrow \infty$, such that 
         $h_{c_{n}}(x) \to h_\infty(x)$ as $n\to \infty$ for some $h_\infty \in C(\R^d)$, where the ODE 
    \begin{equation}
        \dot{x}(t) = h_\infty(x(t)) 
    \end{equation}
    has the origin as its unique globally asymptotically stable equilibrium.
\end{assumption}
We believe that this version is practically attractive as it is often simpler to give a specific sequence $c_n$ such that $h_{c_n}(x)$ converges pointwise, whereas \Cref{asm: BMT} requires that any scaling sequence approaches the same limit. The reason for this is that the original proof of the BMT requires the scaling sequence $\overline{x}(T_m)$ as described in \Cref{sec: outline}. The essence of our new line of argument is that the scaling sequence is monotonically increasing and that it has to diverge to infinity whenever $x_n$ is unstable. This observation is the core idea behind the following generalization. The idea can also be applied to a generalization of the BMT for set-valued recursive inclusions \cite{ramaswamy2017generalization}.

\begin{theorem}
    \label{thm: genBMT}
    Under (A1), (A3)-(A6), $\sup\limits_{n\ge0} \norm{x_n} < \infty \as$.
\end{theorem}
\begin{proof}
The new \Cref{asm: BMT_new} implies that there is a sequence $c_n \nearrow \infty $ such that $h_{c_n}(x) \to h_\infty(x)$ pointwise.  The Arzelà–Ascoli theorem implies that convergence is uniform on compact subsets of $\R^d$. \Cref{lem: liapunov_arg} now also holds for the sequence $c_n$ with the limit $h_\infty$, i.e. there exist $c_0 > 0$ and $T>0$ such that for all initial conditions $x$ on the closed unit ball, $\norm{\phi_{c_n}(x,t)} < \frac{1}{2}$ for $t\in [T,T+1]$ and $c_n>c_0$. Now fix this $T>0$ and follow the construction at the beginning of \Cref{sec: main_results}. 
We will now define a new scaling sequence $s(m)$ to replace the scaling sequence in \eqref{eq: scaling_seq}. Specifically, we define
\begin{equation}
    s(m) \coloneqq \inf_{n\ge 0} \{ c_n \mid \sup_{l\le m} \norm{\overline{x}(T_l)} \le c_n \} \vee 1
\end{equation}
Observe, that $s(m)$ is again by construction monotonically increasing
with the property that $s(m) \ge \norm{\overline{x}(T_m)}.$ 
It follows that the analysis presented in this section holds for this new scaling sequence. It is left to verify the analog to the proof by contradiction for the new scaling sequence.

Fix a sample point where the analogs of \Cref{lem: martingale_convergence} and \Cref{lem: SA_lemma} hold for the new $T>0$ and the new scaling sequence.
Now suppose that $\sup\limits_{m\ge0} \norm{\overline{x}(T_m)} < \infty $ does not hold, then  $ \sup_{l\le m} \norm{\overline{x}(T_l)} \to \infty$.
Observe that $s(m)$ is in essence a subsequence of $c_n$, where some elements get repeated. By construction, $s(m)$ thus inherits the convergence properties from $\{c_n\}$ and
\begin{equation}
    h_{s(m)}(x) \to h_\infty(x)
\end{equation}
uniformly on compact sets. As before, let $x^m(t)$, $t\in [T_m,T_{m+1}]$, be the unique solution to the ODE
\begin{equation}
    \dot{x}(t) = h_{s(m)} (x(t))
\end{equation}
with initial condition $x^m(T_m) = \hat{x}(T_m)$. It follows from the new \Cref{lem: liapunov_arg} that there exists some $M>0$ such that for all initial conditions on the closed unit ball, $\norm{x^m(t)} < \frac{1}{2}$ for $t\in [T,T+1]$ and $m>M$.

Next, we consider those timesteps where $s(m)$ increases, i.e.
let $\{m(k)\}_{k\ge1} \subset \{m \}_{m\ge 1}$, be those timesteps where $s(m(k)) > s(m(k)-1)$. In other words,
\begin{equation}
    s(l) = s(m(k))
\end{equation}
for all $l \in \{m(k), \ldots, m(k+1)-1 \} $. Consider now the first $m(k)$, such that $ m(k)> M$ and $ \sup\limits_{t\in [T_m,T_{m+1}]} \norm{\hat{x}(t) - x^m(t)} < \frac{1}{2}$ for $m > m(k)$ by \Cref{lem: SA_lemma}. As before, the contradiction now arises when we consider the last interval where $s(m(k))$ is used as a scaling factor. We have that
\begin{align} 
 1 <  \frac{\overline{x}(T_{m(k+1)})}{s(m(k))} 
&= \lim\limits_{t \nearrow T_{m(k+1)}}  \norm{\hat{x}(t)} \\ &\le \lim\limits_{t \nearrow T_{m(k+1)}} \norm{x^{m(k+1)-1}(t)} + \lim\limits_{t \nearrow T_{m(k+1)}} \norm{\hat{x}(t) -x^{m(k+1)-1}(t)} \\
&< \frac{1}{2} + \frac{1}{2} = 1.
\end{align}
The contradiction follows since $T_{m(k+1)}$ is by construction a timestep where $s(m)$ strictly increases, hence $\overline{x}(T_{m(k+1)})$ must be greater than the previous scaling. We conclude that $\sup_{m\ge0} s(m) < \infty $ almost surely and the theorem follows as before.
\end{proof}
We will now move on to applications. 

\section{Application I: Decentralized Optimization}
\label{sec: app1}

Consider the unconstrained stochastic optimization problem.
\begin{equation}
    \min\limits_{x \in \R^{d}} F(x)
\end{equation}
with objective
\begin{equation}
    F(x) \coloneqq \int\limits_\Xi f(x;\xi) d\P(\xi)
\end{equation}
for some random function $f:\R^d \times \Xi \rightarrow \R$. Suppose the global optimization variable $x = (x_1, \ldots, x_D) \in \R^d$, with $d \coloneqq \sum_{i=1}^{D} d_i$, is the concatenation of local variables $x_i$, where each local variable is controlled by an associated agent~$i$. To solve the problem, each agent calculates sample partial derivatives $\nabla_{x_i} f(\cdot, \xi)$ to adapt their local variable, i.e. the agents generate sequences $x_n^i$. We assume that the agents do not know the distribution of $\xi$, but at every time slot $n$ they observe \emph{local i.i.d.\ realizations} $\xi_n^i$ of $\xi$. Further, since the agents are physically distributed they have only access to old versions $x^j_{n-\tau_{ij}(n)}$ with AoIPs $\tau_{ij}(n)$. Each agent~$i$ then runs the following SGD iteration starting from some initial $x^i_1 \in \R^{d_i}$:
\begin{equation}
    \label{eq: SGD_iteration}
    x_{n+1}^i = x_n^i - a(n) \nabla_{x_i}f(x^1_{n-\tau_{1i}(n)}, \ldots, x^D_{n-\tau_{Di}(n)}, \xi^i_n) ,
\end{equation}
We will now formulate conditions to apply the theory developed herein. We denote the Hessian matrix of $f$ with respect to $x$ by $\mathbf{H}_x f(x;\xi)$; further, $\norm{\mathbf{H}_x f(x;\xi)}$ denotes matrix norm induced by the considered norm on $\R^d$. 

\begin{assumption}
\label{asm: SGD}
     \begin{enumerate}
        \item $f(x;\xi)$ is twice differentiable in $x$ for $\P$-almost all $y$,
        \item $\sup_{x} \norm{ \Ew{ \mathbf{H}_x f(x;\xi)}} \coloneqq L < \infty$,
        \item $\Ew{\norm{\nabla_x f(x; \xi)}^2} \le K(1+\norm{x}^2$)  for all $x\in \R^d$ for some $K>0$.
     \end{enumerate}
\end{assumption}
\begin{assumption}

\label{asm: SGD2}
    $$ \liminf\limits_{\norm{x} \to \infty} \lambda_{\min} \left(\Ew{ \mathbf{H}_x f(x;\xi)} \right)  > 0, $$
\end{assumption}
where $\lambda_{\min}(\cdot)$ denotes the smallest eigenvalue.
\begin{theorem}
\label{thm: SGD}
Under \Cref{asm: AoI}, \Cref{asm: stepsize}, \Cref{asm: SGD} and \Cref{asm: SGD2} the SGD iterations \eqref{eq: SGD_iteration} converge almost surely to a limit $x_\infty$, such that $\nabla_x F(x_\infty) = 0$.
\end{theorem}
Next, we discuss \Cref{thm: SGD} and its assumptions. 

\subsection{Discussions}

With \Cref{asm: AoI} we assume as before that a dominating variable exists that has some bounded moment based on which \Cref{asm: stepsize} selects the required decaying stepsize. For \Cref{asm: SGD}, we have that each of the following two conditions in addition to \Cref{asm: SGD}.1 and \ref{asm: SGD}.2 imply \Cref{asm: SGD}.3:
\begin{enumerate}
    \item $\xi$ has finite support, i.e. $\abs{\Xi} < \infty$.
    \item $\norm{\nabla_x f(x;\xi)} \le \abs{g(\xi)}(1+\norm{x})$ for some measurable function $g: \Xi \to \R$, such that  $\Ew{g(\xi)^2} < \infty$.
\end{enumerate}
Observe that \Cref{asm: SGD}.3 even allows that the gradient $\nabla_x F(x)$ is unbounded. Such objectives were not covered by the unconstrained optimization setting in \cite{zhou2022distributed}. \Cref{asm: SGD2} is the stability condition that we use to apply the theory developed herein. Some useful conditions that imply \Cref{asm: SGD2} are:
\begin{enumerate}
    \item $F(x) = G(x) + \kappa \norm{x}^2$ for some $L_2$ regularization coefficient $\kappa>0$, where $G(x)$ satisfies \Cref{asm: SGD} and that $G(x) \in o(\norm{x}^2) $.
    \item $F(x) = G(x) + L \norm{x}^2$, where $G(x)$ satisfies \Cref{asm: SGD} with $L>0$ and $G(x)$ is coercive, i.e. $G(x) \to \infty$ as $\norm{x} \to \infty$.
\end{enumerate}
The intuition is that for arbitrary Lipschitz-continuous, coercive functions one needs sufficient regularization to stabilize the SGD iteration in the presence of large delays.

The important class of problems that becomes available for distributed SGD under large delays are quadratic objectives:
 \begin{example}
 Consider
 \begin{equation}
 \label{eq:quadratic_obj}
    f(x;\xi) = x^\top A(\xi) x + b(\xi)^\top x,
 \end{equation}
 where $\Ew{A(\xi)} \in \R^{d\times d}$ is a positive definite matrix with $\Ew{\norm{A(\xi)}^2} < \infty$ 
 and further $\Ew{\norm{b(\xi)}^2} < \infty$. \Cref{asm: SGD} and \Cref{asm: SGD2} hold for \eqref{eq:quadratic_obj} and \Cref{thm: SGD} thus shows \eqref{eq: SGD_iteration} would converge to the global minimum of $F(x) = x^\top \Ew{A(\xi)} x + \Ew{b(\xi)}^\top x$.
 \end{example}

\subsection{Proof of Theorem 3}

We apply the theory developed herein. First, observe that \Cref{asm: SGD}.3 implies that $\nabla_xf(x,y)$ is uniformly integrable for every open neighborhood of some $x\in \R^d$. It thus follows from Lebesgue’s dominated convergence theorem that $\nabla_x F(x) = \Ew{\nabla_x f(x;\xi)}$. Define $h(x) \coloneqq  -\nabla_xF(x)$ and $h_c(x) \coloneqq -\frac{\nabla_x F(cx)}{c}$.

We need two simple auxiliary Lemmas, whose proofs are given in Appendix \ref{sec: appendix}. First, using the boundedness of the expected Hessian it follows that $\nabla_x F(x)$ is Lipschitz-continuous.
\begin{lemma}
\label{lem: lipschitz_in_mean}
    Under \Cref{asm: SGD}.1 and \ref{asm: SGD}.2 it follows that $\nabla_x F(x)$ (hence $h(x)$) is Lipschitz-continuous, with Lipschitz constant $L>0$ from \Cref{asm: SGD}.2.
\end{lemma}

Next, \Cref{asm: SGD2} requires that $F$ acts like a convex function at large scale. Then, the Arzelà-Ascoli theorem guarantees that we can extract a convergent subsequence to satisfy \Cref{asm: BMT_new}. This leads to the second lemma.
\begin{lemma}
\label{lem: scaling_limit} Under \Cref{asm: SGD}.2, \ref{asm: SGD} and \Cref{asm: SGD2}
there is a scaling sequence $c_n \ge 1$ with $c_n \nearrow \infty$, such that
\begin{equation}
    h_{c_n}(x) \to h_\infty(x)
\end{equation}
uniformly on compact sets and $\dot{x}(t) = h_\infty(x(t))$ has the origin as its unique globally asymptotically stable equilibrium.
\end{lemma}

\begin{proof}[Proof of \Cref{thm: SGD}]
Rewrite the SGD iterations as
\begin{align}
    x_{n+1}^i 
    &= x_n^i + a(n)\left[ h(x^1_{n-\tau_{i1}(n)}, \ldots, x^D_{n-\tau_{iD}(n)}) + M^i_{n+1} \right],
\end{align}
with 
\begin{equation}
    M^i_{n+1} \coloneqq \nabla_x F(x^1_{n-\tau_{i1}(n)}, \ldots, x^D_{n-\tau_{i1}(n)})-\nabla_{x_i}  f(x^1_{n-\tau_{i1}(n)}, \ldots, x^D_{n-\tau_{i1}(n)}, \xi^i_n).
\end{equation}
\Cref{lem: lipschitz_in_mean} shows that $h(x)$ Lipschitz continuous and thus satisfies \Cref{asm: lipschitz}. Next, since $\xi^i_n$ are i.i.d., \Cref{asm: SGD}(c) yields that \Cref{asm: noise} holds for $M_{n+1}$. Finally, \Cref{lem: scaling_limit} shows that \Cref{asm: BMT_new} holds. Hence, \Cref{thm: genBMT} shows that $x_n$ is stable almost surely and converges almost surely to a potential sample path-dependent compact connected internally chain transitive invariant set of the ODE $\dot{x}(t) = h(x(t))$. By \Cref{asm: SGD2}, $F$ is coercive and thus itself a Lyapunov function for $\dot{x}(t) = h(x(t))$, hence the only possible invariant set is the set of critic points of $F$.
\end{proof}

\section{Application II: Stochastic Approximation with Momentum}
\label{sec: app2}
 Momentum methods have been extensively studied for gradient-based schemes since the seminal works of Polyak \cite{polyak1964some} and Nesterov \cite{nesterov1983method}. During the last decade, momentum-based methods have seen more and more attention due to their success in deep learning and machine learning \cite{sutskever2013importance,wilson2017marginal}. Notably, momentum has been shown to improve the rate of convergence and has also been shown to help in the avoidance of saddle points \cite{jin2017escape}. 

One of the most studied schemes is SGD with Polyaks heavy-ball momentum known as the stochastic heavy ball (SHB) method. SHB was studied extensively in the last years \cite{Gadat2018heavyball,liu2020improved,sebbouh2021almost,liu2022almost}. See \cite{barakat2021convergence} for a recent detailed analysis of momentum-based gradient schemes using a dynamical systems approach. In \cite{liu2022almost} it was shown for the first time that SHB converges almost surely to a stationary point of non-convex objective functions. 
The assumptions made in \cite{liu2022almost} are the typical assumptions made in a stochastic approximation analysis \cite[Sec. 2]{borkar2009stochastic}. The natural question is therefore whether general stochastic approximations with heavy ball momentum also converge almost surely to an equilibrium.

Stochastic approximation iterations with momentum have also recently seen more attention due to their use in reinforcement learning. \cite{devraj2019matrix} propose a matrix momentum SA method with optimal asymptotic covariance for a class of linear stochastic approximations.  \cite{mou2020linear} studies Polyak-Ruppert averages for linear SA with heavy-ball momentum. In \cite{avrachenkov2022online} a specific one-dimensional SA estimator with heavy-ball momentum was studied for estimating change rates of web pages. In \cite{deb2022gradient, deb2021n} a BMT style Theorem was proven for multi-time scale stochastic approximations. The authors then apply this framework to linear temporal difference learning with heavy-ball momentum which can be viewed as a multi-time scale stochastic approximation. Notably, the above analyses of linear SA with heavy-ball momentum use time-varying momentum parameter $\beta_n \nearrow \infty$.

To our knowledge, the current literature lacks stability and convergence conditions for general stochastic approximation iterations with heavy-ball momentum. We therefore consider the following iteration recalled from \eqref{eq: momentum_intro}:
\begin{equation}
\begin{split}
    x_{n+1} &= x_n + a(n) m_k \\
    m_k &= \beta m_{k-1} + (1-\beta)g(x_k)
\end{split} 
\end{equation}
with $m_0 = 0$, $\beta \in [0,1)$ and $g(x_k) \coloneqq h(x_k) + M_{k+1} $. As mentioned in the introduction, the idea is to rewrite the moving average of the past drift terms using a determinism AoI sequence: 
\begin{equation}
\label{eq: momentum_intro_recalled}
    x_{n+1} = x_n + a(n) (1-\beta)\left[\sum_{i=n-\tau(n)}^n \beta^{n-i} g(x_i) \right] +  a(n) (1-\beta)\left[\sum_{i=1}^{n-\tau(n)-1}\beta^{n-i} g(x_i) \right],
\end{equation}
with deterministic AoI $\tau(n) = \lceil \frac{n}{\log(n+1)}\rceil$. The iteration can now be studied using the tools presented in this paper. This is possible since the second summation is merely a vanishing error. This leads to our last result.

\begin{theorem}[BMT for SA with heavy-ball momentum]
    \label{thm: momentum}
    Suppose that 
    \begin{itemize}
    \item $h$ is Lipschitz continuous and satisfies \Cref{asm: BMT} (or \Cref{asm: BMT_new}),
    \item $M_{n+1}$ is a zero mean square integrable martingale Difference sequence with \\$\Ew{\norm{M_{n+1}}^2 \mid \cF_n} \le K^2 (1+\norm{x_n})$, where $\cF_n \coloneqq \sigma(x_1, M_{2}, \ldots, M_{n} )$.
    \item $a(n) \in \cO(\frac{1}{n})$
\end{itemize}
then the stochastic heavy-ball iteration \eqref{eq: momentum_intro} is stable almost surely.
\end{theorem}
\begin{corollary}
    \label{cor: momentum}
    Under the assumptions of \Cref{thm: momentum}, the stochastic heavy-ball iteration \eqref{eq: momentum_intro} converges almost surely to a potential sample path-dependent compact connected internally chain transitive invariant set of $\dot{x}(t) = h(x(t))$.
\end{corollary}

\begin{remark}
    The following proof can be extended to every $a(n)$ that is not summable but square summable. For this choose $\tau(n) = \lceil \frac{n}{\sum_{k=1}^{n} a(n) }\rceil$
\end{remark}

\subsection{Proof of Theorem 4 and Corollary 5}

Define $c(n) \coloneqq \sum_{i=n-\tau(n)}^n \beta^{n-i}$ and define the tail error
\begin{equation}
    \delta_n \coloneqq \frac{1}{c(n)}  \sum_{i=1}^{n-\tau(n)-1} \beta^{n-i} g(x_i).
\end{equation}
It is now not difficult to show that
\begin{equation}
\label{eq:c(n)_properties}
    \sup_{n\ge0} c(n) = \lim\limits_{n\to \infty} c(n) = \frac{1}{1-\beta} \text{ and } \sup_{n\ge 1} \frac{1}{c(n)} \le 1.
\end{equation}
We can further show that the tail error vanishes almost surely. The proof is given in the Appendix.
\begin{lemma}
    \label{lem: momentum_expgrowth}
    $\norm{\delta_n} \in o(1).$
\end{lemma}

Next, write \eqref{eq: momentum_intro_recalled} as
\begin{equation}
    \label{eq: momentum_analysis_iteration}
    x_{n+1} = x_n + \tilde{a}(n) \left[ h(x_n) + \tilde{M}_{n+1} + e_n + \delta_n\right].
\end{equation}
with $\tilde{a}(n) \coloneqq a(n) (1-\beta)c(n)$ and
\begin{equation}
    e_n \coloneqq \frac{1}{c(n)} \sum_{i=n-\tau(n)}^{n-1} \beta^{n-i} \left( h(x_i) - h(x_n) \right), \quad \tilde{M}_{n+1} \coloneqq \frac{1}{c(n)}\left( \sum_{i=n-\tau(n)}^n  \beta^{n-i} M_{i+1}\right).
\end{equation}
Equation \eqref{eq: momentum_analysis_iteration} can now be analyzed with  the tools and techniques presented in \Cref{sec: main_results}. Specifically, \eqref{eq: momentum_analysis_iteration} is a SA iteration with Martingale difference noise $\tilde{M}_{n+1}$, an error $e_n$ due to the deterministic AoI sequence and the additional error $\delta_n \in o(1)$.

\begin{proof}[Proof of \Cref{thm: momentum}]
    First, observe that $\tilde{M}_{n+1}$ is also a zero-mean martingale difference sequence with respect to $\cF_n$. Further, using the Lipschitz continuity of $h$, it follows that
    \begin{equation}
        \label{eq: momentum_zeros_ineq}
        \norm{e_n} \le L \sum_{i=n-\tau(n)}^{n-1} \beta^{n-i} \norm{x_n-x_i}. 
    \end{equation}
    Analogously to \Cref{sec: main_results}, we will now derive recursive inequalities in $L_2$ and in norm. We will illustrate this for the $L_2$ case. Define $\tilde{e}_n \coloneqq \sum_{i=n-\tau(n)}^{n-1} \beta^{n-i} \norm{x_n-x_i}$.

    With \eqref{eq: momentum_analysis_iteration} and the unit growth property of $h$ we have that
    \begin{align}
        \label{eq: momentum_first_ineq}
        \norm{x_n - x_i} \le \sum_{j=i}^{n-1} \tilde{a}(j) K(1+\norm{x_j}) + \norm{\sum_{j=i}^{n-1} \tilde{a}(j) \tilde{M}_{j+1}} + \sum_{j=i}^{n-1} \tilde{a}(j) L \tilde{e}_j + \sum_{j=i}^{n-1} \tilde{a}(j) \norm{\delta_j}.
    \end{align}
    Next, use the $L_2$ bound of $M_{n+1}$ for each term in $\tilde{M}_{j+1}$ and expand the resulting terms:
    \begin{equation}
        \Ew{\norm{\tilde{M}_{j+1}}^2}^\frac{1}{2} \le \frac{K}{1-\beta}\left(1+ \Ew{\norm{x_j}^2}^\frac{1}{2} + \sum_{i=j-\tau(j)}^{j-1} \beta^{j-i} \Ew{\norm{x_j-x_i}^2}^\frac{1}{2}\right)
    \end{equation}
    Define $\tilde{e}^{L_2}_n \coloneqq \sum_{i=n-\tau(n)}^{n-1} \beta^{n-i} \Ew{\norm{x_n-x_i}^2}^{\frac{1}{2}} $. Then 
    \begin{equation}
        \Ew{\norm{x_n - x_i}^2}^\frac{1}{2} \le \sum_{j=i}^{n-1} \tilde{a}(j) K_1(1+ \Ew{\norm{x_j}^2}^\frac{1}{2})
        + \sum_{j=i}^{n-1} \tilde{a}(j) K_2 \tilde{e}^{L_2}_j + \sum_{j=i}^{n-1} \tilde{a}(j) \Ew{\norm{\delta_j}^2}^\frac{1}{2}.
    \end{equation}
    for some constants $K_1, K_2 >0$.
    Thus, we have that
    \begin{align}
       \tilde{e}^{L_2}_n &\le \sum_{i=n-\tau(n)}^{n-1} \beta^{n-i}  \left(\sum_{j=i}^{n-1} \tilde{a}(j) K_1(1+ \Ew{\norm{x_j}^2}^\frac{1}{2})
        + \sum_{j=i}^{n-1} \tilde{a}(j) K_2 \tilde{e}^{L_2}_j + \sum_{j=i}^{n-1} \tilde{a}(j) \Ew{\norm{\delta_j}^2}^\frac{1}{2}\right) \\
        &\le K_1 (1+ \sup_{j\le n-1}\Ew{\norm{x_j}^2}^\frac{1}{2}) \left(\sum_{i=n-\tau(n)}^{n-1} \beta^{n-i} \left(\sum_{j=i}^{n-1} \tilde{a}(j) \right)\right) \nonumber \\
        &+ K_2 \left(\sum_{i=n-\tau(n)}^{n-1}  \sum_{j=i}^{n-1} \beta^{n-i} \tilde{a}(j) \tilde{e}^{L_2}_j \right) + \left(\sum_{i=n-\tau(n)}^{n-1}  \sum_{j=i}^{n-1} \beta^{n-i} \tilde{a}(j) \Ew{\norm{\delta_j}^2}^\frac{1}{2}\right), \\
        &\le \frac{K_1}{1-\beta} (1+ \sup_{j\le n-1}\Ew{\norm{x_j}^2}^\frac{1}{2}) \left(\sum_{j=n-\tau(n)}^{n-1}  \tilde{a}(j) \right) \nonumber\\
        &+ \frac{K_2}{1-\beta} \left(\sum_{j=n-\tau(n)}^{n-1}  \tilde{a}(j) \tilde{e}^{L_2}_j \right) + \frac{1}{1-\beta}\left(\sum_{j=n-\tau(n)}^{n-1}  \tilde{a}(j) \Ew{\norm{\delta_j}^2}^\frac{1}{2}\right),
    \end{align}
    where the last inequality follows from \eqref{eq:c(n)_properties}. 
   The similarity to the recursive inequality \eqref{eq: error_iterative} should be clear. The same inequality will hold when the sequences are scaled by the monotone scaling sequence used in \Cref{sec: main_results}. We can thus apply our new Gronwall-type inequality \Cref{lem: new_gronwall}. As $\Ew{\norm{\tilde{e}_n}^2}^{\frac{1}{2}} \le  \tilde{e}^{L_2}_n$, evaluating the recursion for $\tilde{e}^{L_2}_n$ leads to the required $L_2$ bound for $\tilde{e}_n$. Furthermore, from \eqref{eq: momentum_first_ineq} we also arrive at a recursive inequality $\tilde{e}_n$ itself. The stability analysis in \Cref{sec: main_results} then goes through using the $L_2$ bound and the norm bound for $\tilde{e}_n$ along exactly the same line of argument. The tail error $\delta_n$ only adds neglectable terms in the analysis. We omit the details to avoid redundancies.
\end{proof}

\begin{proof}[Proof of \Cref{cor: momentum}]
    \Cref{thm: momentum} shows that there is sample path dependent radius $R>0$, such that $x_n \in \cB_R(0)$ for all $n\ge 1$. It thus follows that $\tilde{e}_n$ as defined in 
 \Cref{thm: momentum} satisfies $\tilde{e}_n \le \frac{2B}{1-\beta}$. Further, using stability the martingale convergence theorem yields that $\sum_{j=1}^{n} \tilde{a}(j) \tilde{M}_{j+1}$ converges almost surely. \Cref{lem: aoi_bound} and \Cref{lem: step_size_AoI_sum} then imply that $\sum_{j= i}^{n-1} \tilde{a}(j) \to 0$ for all $ n - \tau(n) \le i \le n-1 $. It follows that the right-hand side of \eqref{eq: momentum_first_ineq} converges to zero almost surely and thus \eqref{eq: momentum_zeros_ineq} shows that $\norm{e_n} \to 0$ almost surely. Equation \eqref{eq: momentum_analysis_iteration} is therefore a standard SA iteration with vanishing additive error and the corollary follows.
\end{proof}

\appendix

\section{Background}
\begin{lemma}[Discrete Gronwall Inequality {\cite[Appendix B]{borkar2009stochastic}} ]
    \label{gronwall_ineq}
    Let $x_n$, $a_n$ nonnegative (respectively positive) sequences and $C,L>0$ scalars such that for all $n$, 
    $$x_{n+1} \le C + L \sum_{m=0}^n a_m x_m,$$
    then $x_{n+1} \le  C \exp(L \sum_{m=0}^{n} a_m)$.
\end{lemma}

\begin{lemma}[Martingale Convergence Theorem {\cite[Appendix C]{borkar2009stochastic}} ]
    \label{martingale_conv}
    Let $\{(X_n, \cF_n)\}_{n\ge1}$ be a martingale, if $\Ew{\norm{X_n}^2}<\infty$ for all $n\ge1$, then $X_n$ converges almost surely on the set
    $$ \sum_{n \ge 1} \Ew{\norm{X_{n+1} - X_n}^2 \mid \cF_n} < \infty.$$
\end{lemma}

\section{Missing proofs}
\label{sec: appendix}

\begin{proof}[Proof of \Cref{lem: aoi_bound}]
We present a proof for the difficult case, $p \in (0,1]$. The proof for $p>1$ follows from the first Borel-Centelli lemma and can be found in \cite{redder2022_3DPG} in the context of distributed reinforcement learning.

Fix, $\varepsilon \in (0,1)$. Using the stochastic dominance property it follows that
\begin{align}
    \sum_{n\ge1} \Pr{\tau(n^{\frac{1}{p}}) > \varepsilon n^{\frac{1}{p}}} \le \sum_{n\ge1} \Pr{\overline{\tau} > \varepsilon n^{\frac{1}{p}}} = \sum_{n\ge1} \Pr{\frac{1}{\varepsilon^p}\overline{\tau}^p >  n} \le 1 +  \frac{1}{\varepsilon^p}\Ew{\overline{\tau}^p}.
\end{align}
The last inequality can be found in \cite[Theorem 3.2.1]{chung2001course}.
Since,  $\Ew{\tau^p}  < \infty$ it follows from the first Borel-Cantelli lemma that $\tau(n^{\frac{1}{p}}) \le \varepsilon n^{\frac{1}{p}} $
for all $n\ge N(\varepsilon)$ with a sample path dependent constant $N(\varepsilon) \ge 0.$
In other words,
\begin{equation}
    \label{eq:powers}
    \tau(n) \le \varepsilon n , \qquad \text{for } n = k^{\frac{1}{p}} \text{ with} k\ge N(\varepsilon).
\end{equation}
Next, consider two subsequent integers $n'$ and $n''$ that satisfy \eqref{eq:powers}.
Specifically, let $k\ge N(\varepsilon) $ and fix $n' = k^{\frac{1}{p}}$ and $n'' = (k+1)^{\frac{1}{p}}$.
By the unit growth of the AoIP it follows that the AoI of the time steps between $n'$ and $n''$ satisfy
\begin{equation}
\label{eq:bound_1}
    \tau(n'+i) \le \tau(n') +i \le  \varepsilon n' + i 
\end{equation}
for every $i \in \{0, \ldots, n'' - n'\}$, Further, $i \le C k^{\frac{1}{p} - 1}$ for a constant $C >0$ that only depends on $p$.
Since $k = (n')^p$, it follows that
\begin{equation}
\label{eq:bound_2}
     i \le  C  (n')^{1-p} \le  C  (n'+i)^{1-p} 
\end{equation}
By combining \eqref{eq:bound_1} and \eqref{eq:bound_2} it follows that
\begin{align}
    \tau(n'+i) \le \varepsilon n' + i  
    = \varepsilon (n' + i) + (1-\varepsilon) i \le  \varepsilon(n' + i) + (1-\varepsilon) C (n'+i)^{1-p} 
\end{align}
Since this holds for all pairs $n'$, $n''$ with $i \in \{0, \ldots, n'' - n'\}$, it follows that
\begin{equation}
    \tau(n) \le \varepsilon n + (1-\varepsilon)  C n^{1-p}
\end{equation}
for all $n \ge N(\varepsilon)^{\frac{1}{p}}$. We have therefore shown that
\begin{equation}
    \label{eq: initial_limsup}
    \Pr{\tau(n) > \varepsilon n + (1-\varepsilon)C n^{1-p} \io } = 0.
\end{equation}
for every $\varepsilon \in (0,1)$.
Now fix $\varepsilon' \in (0,1)$ and let $\varepsilon = \frac{1}{2} \varepsilon'$.
Then $(1-\varepsilon)C n^{1-p} \le \varepsilon n$ for all $n \ge N$ for some $N \in \N$. Hence,
\begin{equation}
    \{ \tau(n) > \varepsilon n + (1-\varepsilon)C n^{1-p} \} \supset \{ \tau(n) > 2\varepsilon n\}
\end{equation}
for $n \ge N$.
By definition of the limit supremum, it then follows from \eqref{eq: initial_limsup} that
\begin{align}
    \Pr{\tau(n) > \varepsilon'n \io } 
    &\le \Pr{\tau(n) > \varepsilon n + (1-\varepsilon)C n^{1-p} \io } = 0.
\end{align}
\end{proof}

\begin{proof}[Proof of case $p>1$ in \Cref{lem: step_size_AoI_sum}]
W.l.o.g. consider $p \in (1,2)$ and let $a(n) = a n^{-\frac{1}{p}}$.
We have
\begin{align}
    \sum_{k=n-\tau(n)}^{n} a(k) &\le  a(n-\tau(n))^{-\frac{1}{p}} + a \int_{n-\tau(n)}^n t^{-\frac{1}{p}} dt \\ &= a(n-\tau(n))^{-\frac{1}{p}} + \frac{a}{\left(1-\frac{1}{p}\right)} \left( n^{1 -\frac{1}{p}} - (n-\tau(n))^{1-\frac{1}{p}} \right) \\
    &= a(n-\tau(n))^{-\frac{1}{p}} + \frac{a}{\left(1-\frac{1}{p}\right)} n^{1 -\frac{1}{p}}\left(1  - \left(1-\frac{\tau(n)}{n}\right)^{1-\frac{1}{p}} \right) \label{eq: sum_bound}
\end{align}
Bernoulli's inequality for real negative exponents \cite{wiki:Bernoulli's_inequality} shows that
\begin{align}
    \left(1-\frac{\tau(n)}{n}\right)^{1-\frac{1}{p}} &= \left(1-\frac{\tau(n)}{n}\right)  \left(1-\frac{\tau(n)}{n}\right)^{-\frac{1}{p}} \\ &\ge \left(1-\frac{\tau(n)}{n}\right)\left(1+ \frac{1}{p}\frac{\tau(n)}{n}\right) \ge \left( 1-\left(1-\frac{1}{p}\right)\frac{\tau(n)}{n} \right) \label{eq: bernoulli}
\end{align}
Hence, \eqref{eq: sum_bound} and \eqref{eq: bernoulli} show that
\begin{equation}
    \sum_{k=n-\tau(n)}^{n} a(k) \le a(n-\tau(n))^{-\frac{1}{p}} + an^{-\frac{1}{p}}\tau(n)
\end{equation}
Finally, by (A1) there is a random variable $\overline{\tau}$ with $\tau(n) \lest \overline{\tau}$ and $\Ew{\overline{\tau}^p}< \infty$. \Cref{lem: aoi_bound} then shows $\Pr{n^{-\frac{1}{p}} \tau(n) > \varepsilon  \io} = 0$ for every $\varepsilon \in (0,1)$. It therefore follows from Lemma used to proof the case $p\in (0,1]$ that $n^{-\frac{1}{p}} \tau(n) \to 0 \as$
\end{proof}

\begin{proof}[Proof of \Cref{lem: new_gronwall}]
    Let $n\ge 0$, then for all $m\le n$,
    \begin{equation}
        y_m \le c_n B + C \sum_{k=0}^{n-1} a_k y_k. 
    \end{equation}
    The traditional discrete Gronwall inequality,  \Cref{gronwall_ineq}, thus implies that
    \begin{equation}
    \label{eq: initial_gronwall}
        y_n \le c_n B e^{C T(n)}.
    \end{equation}
    Since $\sum_{k=n - \tau_n}^{n-1} a_k \to 0$ it follows that $N < \infty$. It now follows that
    \begin{align}
        y_{N+1} &\le c_{N+1}b_{N+1} + C \sum_{k=N + 1 - \tau_{N+1}}^{N} a_k y_k \\
        &\le c_{N+1}b_{N+1} +  c_N B e^{C T(N)} C \sum_{k=N + 1 - \tau_{N+1}}^{N} a_k \\
        &\le c_{N+1}b_{N+1} + c_N B(e^{CT(N)} - 1) \le c_{N+1} Be^{CT(N)}.
    \end{align}
    The second inequality uses \eqref{eq: initial_gronwall} and that both $c_n$ and $T(n)$ are increasing. The third inequality applies the definition of $N$. The last inequality uses again that $c_n$ is increasing and that $b_{N+1} \le B$.
    It now follows by induction that 
    \begin{equation}
        y_n \le c_n B e^{CT(N)} 
    \end{equation}
    for all $n\ge 0$. By using this inequality in the initial inequality we obtain
    \begin{align}
        y_n &\le b_nc_n + C \sum_{k=n - \tau_n}^{n-1} a_k y_k, \\
        &\le b_n c_n + C \sum_{k=n - \tau_n}^{n-1} a_k  c_k B e^{CT(N)} \\
        &\le c_n \left( b_n + B C e^{C T(N)} \left(\sum_{k=n - \tau_n}^{n-1} a_k \right) \right),  
    \end{align}
    which proves the desired inequality. 
\end{proof}

\begin{proof}[Proof of \Cref{lem: lipschitz_in_mean}]
    Let $x',x'' \in \R^d$. Since $f$ is twice differentiable in $x$ by \Cref{asm: SGD}.1, we can apply the mean value theorem (MVT) for vector-valued multivariate functions to the first coordinate of $f$, \cite{mcleod1965mean}. This yields
    \begin{equation}
    \label{eq: mean_value}
        \nabla_x f(x'; y) - \nabla_x f(x''; y) = \mathbf{A}(x',x'',y) \left( x' - x''\right) ,
    \end{equation}
    for $\P$-almost all $y$ with a matrix 
        $\mathbf{A}(x',x'',y) \in \convh \{ \mathbf{H}_x f(x,y) \mid x \in \overline{x'x''} \},$
    i.e., $\mathbf{A}(x',x'',y)$ is a convex combination of $\mathbf{H}_xf(x,y)$ evaluated along the line $\overline{x'x''}$.
    Now take the expected value and norm on both sides of \eqref{eq: mean_value}, then
    \begin{equation}
        \norm{\Ew{\nabla_x f(x'; y)} - \Ew{\nabla_x f(x''; y)}} \le  \norm{\Ew{\textbf{A}(x',x'',y)}}\norm{x - x'},
    \end{equation}
    Finally, apply \Cref{asm: SGD}.2 to conclude that
    \begin{equation}
        \label{eq: Lip_2_ineq}
        \norm{\Ew{\nabla_x f(x'; y)}  - \Ew{\nabla_x f(x''; y) }}  \le   L \norm{x' - x''}.
    \end{equation}
\end{proof}

\begin{proof}[Proof of \Cref{lem: scaling_limit}]
Consider a scaling sequence $c_m \ge 1$ with $c_m \nearrow \infty$.
The Lipschitz continuity of $h(x)$ yields that $\{ h_{c_m}(x) : n\ge 1\}$ is an equicontinuous, pointwise bounded family of continuous functions. A general version of the Arzelà-Ascoli theorem, see e.g. \cite{billingsley2013convergence}, now yields that the family is relatively compact in the subspace consisting of continuous functions, equipped with the topology of compact convergence. In other words, there exists a subsequence $m(n)$, such that $h_{c_{m(n)}}(x)$ convergence compactly to some limit $h_\infty(x)$. Define $c_n \coloneqq c_{m(n)}$. It is left to show that $\dot{x}(t) = h_\infty(x(t))$ has the origin as its unique globally asymptotically stable equilibrium.

Let $x, z \in \R^d$. Similar to \Cref{lem: lipschitz_in_mean}, the mean value theorem yields that
    \begin{equation}
        \frac{\nabla_x F(c_n x)}{c_n} - \frac{\nabla_x F(z)}{c_n} = \mathbf{A}(c_n x,z) \left( x - \frac{z}{c_n}\right) ,
    \end{equation}
    with 
        $\mathbf{A}(x',x'') \in \convh \{ \Ew{\mathbf{H}_x f(x,y)} \mid x \in \overline{x'x''} \},$ where we already exchanged the order of the convex combination and the expectation.
Thus, 
\begin{equation}
    \label{eq: limit_h_inf}
    h_\infty(x) \coloneqq - \lim\limits_{n\to \infty} \mathbf{A}(c_n x,z) x
\end{equation}
for every $z\in \R^d$. \Cref{asm: SGD2} now requires that
\begin{equation}
    \sup \{ \inf  \{ \lambda_{\min} \left(\Ew{ \mathbf{H}_x f(x;\xi)} \right) : \norm{x} > r \}: r>0\}   > 0,
\end{equation}
which implies that there is some $r>0$ and some $\varepsilon>0$, such that $\lambda_{\min} \left(\Ew{ \mathbf{H}_x f(x;\xi)} \right) > \varepsilon$ for $\norm{x} > r$. Thus for all $\norm{z} > r$ and $\norm{c_n x} > r$ it follows that $\mathbf{A}(c_n x,z)$ is a convex combination of positive definite matrices with the smallest eigenvalue greater than $\varepsilon$. Elementary properties of positive definite matrices thus yield that that $\lambda_{\min} \left(\mathbf{A}(c_n x,z) \right) > \varepsilon $ for $\norm{z} > r$ and $\norm{c_n x} > r$. Finally, we conclude from \eqref{eq: limit_h_inf} that for every $x\in \R^d$, $h_\infty(x) = - \lim\limits_{n\to \infty} \mathbf{A}(x) x$ for some matrix $\mathbf{A}(x) \in \R^{d\times d}$ with $\lambda_{\min} \left(\mathbf{A}(x) \right) > 0 $. Thus $\dot{x} (t) = h_\infty(x(t))$ is globally asymptotically stable to the origin by LaSalle's invariance principle.
\end{proof} 

\begin{proof}[Proof of \Cref{lem: momentum_expgrowth}]
    Recall that the heavy ball iteration is given by
    \begin{equation}
        x_{n+1} = x_n + a(n) (1-\beta)\left[\sum_{i=1}^n \beta^{n-i} g(x_i) \right],
    \end{equation}
    with $g(x_n) \coloneqq h(x_n) + M_{n+1} $.
    Then using Gronwalls inequality \Cref{gronwall_ineq}, $a(n) \in \cO(\frac{1}{n})$ and $\Ew{\norm{M_{n+1}}^2} \le K^2(1+ \Ew{\norm{x_n}^2})$, it follows that 
    \begin{equation}
        \Ew{\norm{x_n}^2} \le C n
    \end{equation}
    for some constant $C>0$. With this, it follows from the martingale convergence theorem \Cref{martingale_conv} that
    $\sum_{n=1}a(n) \frac{M_{n+1}}{n}$ converges almost surely. Since $\sum_{n\ge 1} a(n) = \infty$, it follows that $M_{n+1} \in o(n)$.
    The lemma now follows from
\begin{align}
    \norm{\delta_n} &\le \frac{(1-\beta)}{c(n)}  \sum_{i=1}^{n-\tau(n)-1} \beta^{n-i}\norm{g(x_i)} 
     \in (1-\beta) o\left( \frac{(1-\beta)}{c(n)}  \sum_{i=1}^{n-\tau(n)-1} \beta^{n-i} n \right) = o(1).
\end{align}
\end{proof}


\section*{Acknowledgments.}
Adrian Redder was supported by the German Research Foundation - SFB901.
An important part of this work was done while Adrian Redder visited Arunselvan Ramaswamy at Karlstad University.

\bibliographystyle{apalike}
\bibliography{references}

\end{document}